\documentclass{article}
\usepackage{graphicx} 

\usepackage[a4paper, margin=2cm]{geometry}
\usepackage{amsmath,amssymb}
\usepackage{hyperref}
\usepackage{enumitem}
\usepackage{martonsdefs}
\usepackage{martonsthms}
\usepackage{pgfplots}
\pgfplotsset{compat=1.18}
\usepackage{caption}
\usepackage{subcaption}
\usepackage{tikz}
\usetikzlibrary{arrows.meta,calc}
\usetikzlibrary{intersections}
\usetikzlibrary{patterns}
\usepackage{array,booktabs}
\usepackage{float} 
\usepackage[dvipsnames]{xcolor} 
\definecolor{vmaxcolor}{RGB}{0,100,0} 
\newcolumntype{M}[1]{>{\arraybackslash}m{#1}} 
\usepackage{bbm}
\usepackage{graphicx}
\usepackage{mathtools}
\usepackage{tabularx}
\usepackage{verbatim}
\usepackage[T1]{fontenc}
\usepackage[outline]{contour}

\newbool{pics}

\setbool{pics}{true}

\ifbool{pics}{\newenvironment{condpics}{}{}}{\newenvironment{condpics}{\comment}{\endcomment}}

\newcommand*{\vmax}{v_\text{max}}
\newcommand*{\vel}{v_{\text{E--L}}}

\makeatletter
\newcommand{\hathat}[1]{%
\begingroup%
  \let\macc@kerna\z@%
  \let\macc@kernb\z@%
  \let\macc@nucleus\@empty%
  \hat{\raisebox{.2ex}{\vphantom{\ensuremath{#1}}}\smash{\hat{#1}}}%
\endgroup%
}
\makeatother

\title{The fastest way through a traffic light}
\author{M\'arton Bal\'azs\footnotemark[1], Edward Crane\footnotemark[1], and Alexander Tallis\footnotemark[1]\thanks{University of Bristol}}
\date{\today}

\begin{document}

\maketitle

\begin{abstract}
We give a rigorous solution of the problem of minimising the expected delay caused by a red traffic light on a road journey. The problem incorporates constraints on maximum speed, acceleration and braking rates, and depends on the distribution of the remaining time until the light turns green. In the cases of Uniform and Exponential distributions, we give a complete characterisation of all possible combinations of phases of maximum acceleration, maximum speed, maximum braking, following an Euler-Lagrange curve, and standing stationary at the traffic light, which can make up an optimal solution. In the Exponential case there is a surprise: the optimal solution sometimes involves a discontinuous jump from one positive braking rate to the maximum allowed braking rate. The key technique is to write the problem in terms of a two-dimensional pressure integral, so that the problem becomes analogous to filling a tank with a given quantity of liquid.
\end{abstract}

\medskip

\noindent
\textbf{MSC2020 subject classification:} 49K30, 49K45.

\section{Introduction}

\subsection{The problem}

Imagine that you drive around a corner and spot a red traffic light in front of you, a positive distance $d$ away. At this moment your velocity is $v_0$.  You are very law-abiding, so you will always avoid driving through a red traffic light. Your car has a maximum acceleration rate $\alpha$, which does not depend on your speed. It has a maximum braking rate $\beta$, which also does not depend on your speed.
Fortunately for you, $d \ge v_0^2/(2\beta)$, so if you brake hard immediately away you can avoid passing the traffic signal on red. You are in a hurry and wish to reach your destination as early as possible.  The street has a speed limit $\vmax$ which you will of course avoid exceeding, being law-abiding. No other vehicles are between you and the traffic light, but there is another driver close behind you, so you cannot choose to reverse without causing a road-rage incident. That is to say, $v_\text{min} = 0$. Your destination is a reasonably large distance $L$ beyond the traffic light, and there are no further traffic lights between the one traffic light that you can see and your destination. We assume that $L \ge \vmax^2/(2\alpha)$, so that even if you have to stop at the traffic light you will have time to accelerate all the way to the speed limit $\vmax$ before reaching your destination. (To keep things simple, suppose that it is acceptable to be travelling at speed $\vmax$ at the moment you reach your destination, so that it functions like the finish line in a race.) We also assume that once the light turns green it is guaranteed to stay green for long enough that you will certainly be able to pass it during the green phase, by accelerating at rate $\alpha$ and then moving with velocity $\vmax$. (As far as our optimisation problem below is concerned, this is the same as assuming it will stay green forever once it has turned green.)
\begin{quote}How should you accelerate or brake, to minimise the expectation of your arrival time, without taking any risk of breaking the law? 
\end{quote}
The answer depends on your belief about the distribution of the random time $T$ remaining until the traffic light turns green. Let us look at a few specific cases and a more general case.

\begin{enumerate}
 \item The traffic light has a countdown display next to it which indicates $T$. This case is simple to solve and we will not consider it further, but note that optimal solutions exist that can be decomposed into phases where either the velocity is $0$ or $\vmax$, or the acceleration is $\alpha$ or $-\beta$. We will see that in a much more general case the optimal solutions may also be decomposed into finitely many phases of these four types and one further type, an \emph{Euler--Lagrange curve}. 
 \item You are sure that the light will turn green before a finite time $q >0$, and even if you travel as fast as allowed by the constraints you cannot pass the traffic light before time $q$. This case is also trivial, so for the rest of the paper we exclude it by assuming  
 \begin{equation}
 \begin{cases}
 q \le (\vmax - v_0)/\alpha \;\text{ and }\;  d < q(v_0+q\alpha/2), \text{ or}
\\
  q \ge (\vmax - v_0)/\alpha \;\text{ and }\; d <  (\vmax -v_{0})(\vmax+v_0)/2\alpha + \vmax(q - (\vmax - v_0)/\alpha). 
 \end{cases}
\label{eq:trafficlightisnear}
 \end{equation}
\item\label{it:uni} You believe that the time until the light turns green is a random variable with a uniform distribution, say $U([0,q])$.  This is consistent with a traffic light that always stays red for exactly $q$ seconds, when we don't know how much of the red phase has elapsed when we first observe it. 
\item\label{it:exp} You believe that the traffic light will turn green after an exponentially distributed time with mean $1/\lambda$. In this case there is no upper bound on $T$ so we take $q = \infty$. 
\item\label{it:renewal} You believe that the lengths of the periods when the light is red are i.i.d.~random variables with a distribution that you know, but you don't know how long the light has already been red when you first observe it. In this case it is reasonable to model the remaining time until the light turns green as a random variable distributed like the \emph{excess time} (or residual time) until the next renewal epoch in a  stationary renewal process. Cases \ref{it:uni} and \ref{it:exp} are both special cases of this one.
\end{enumerate}

In this paper we prove the existence of a unique optimal trajectory in case \ref{it:renewal}. We also work out the exact optimal trajectory in cases \ref{it:uni} and \ref{it:exp}.

\subsection{The solution}

We first give a brief, informal description of the solution for cases \ref{it:uni} and \ref{it:exp}. The optimal solution is a velocity profile $v(t)$, which is best described as a 2-dimensional tank-and-water picture. These are illustrated in Figures \ref{fig:uniformpressure} for the Uniform case and \ref{fig:exppressure} for the Exponential case. The tank has the following boundaries:
\begin{itemize}
 \item On the left, below $v(0)=v_0$, a line of slope $-\be$ until velocity 0. This is slamming the brakes right at the start.
 \item On the bottom, the line $v=0$, as we never reverse.
 \item When $q$ is finite, there is a boundary on the right, along the line $t=q$, since the optimal profile is only valid until the light goes green at time $T$, which never exceeds $q$.
 \item On the left, above $v(0)=v_0$, a line of slope $\al$, as this is the maximum acceleration.
 \item On top, the line $v = \vmax$, as this is the maximum permitted speed.
\end{itemize}

We will fill this tank up with water. The total volume of the water, $\int_0^qv(s)\,\text ds$, is the distance travelled up to the light, which is $d$. The catch is that the level lines of the water are not necessarily horizontal; we will derive their exact shape later in \eqref{eq:velexpression}. Just imagine the whole tank placed in a funny gravitational field.
\begin{itemize}
 \item For the Uniform case, the level lines are straight lines of slope $-\al$. (In this case, the gravitational field is the one we used to and the tank is tilted, as shown in Figure~\ref{fig:uniformpressure}.)
 \item For the Exponential case, the level lines are decreasing exponential curves getting steeper in time. Under some parameter values the slope can exceed $-\be$, which violates the constraints. In this case, the exponential level lines suddenly switch to a straight line of slope $-\be$ at a particular velocity value \(v_c\). That is, at this velocity we suddenly slam the brakes at full force. Notice that the switch is not smooth: there is a discontinuity in the deceleration when it happens.
\end{itemize}

To get the optimal trajectory $v(t)$, fill this tank up with volume $d$ of water, and follow the upper boundary of the liquid surface. The curves are marked with blue dash-dotted arrows.

\begin{condpics}
\begin{figure}[ht]
  \centering
  \begin{tikzpicture}[rotate=34.3]
    \begin{axis}[
      axis lines=left,
      xmin=0, xmax=2.2,
      ymin=0, ymax=1.4,
      ticks=none,
      width=0.8\linewidth,
      height=0.4\linewidth,
      xlabel={$t$},
      ylabel={$v$},
      xlabel style={
        at={(axis description cs:0.5,-0.08)},
        rotate=-35
      },
      ylabel style={
        at={(axis description cs:-0.04,0.5)},
        rotate=-125
      },
      clip=false
    ]

      \node[below, rotate=-35] at (axis cs:0.30,0) {$t_1$};
      \node[below, rotate=-35] at (axis cs:0.60,0) {$t_2$};  
      \node[below, rotate=-35] at (axis cs:0.70,0) {$t_3$};  
      \node[below, rotate=-35] at (axis cs:2.00,0) {$t_4$};  
      \node[below, rotate=-35] at (axis cs:2.10,0) {$q$};
      \node[left,  rotate=-35] at (axis cs:0,1.30) {$\vmax$};
      \node[left,  rotate=-35] at (axis cs:0,0.60) {$v_0$};

\begin{scope}
  \path[clip]
    (axis cs:0,0.60) --
    (axis cs:0.70,1.30) --
    (axis cs:2.10,1.30) --
    (axis cs:2.10,0) --
    (axis cs:0.3,0) --
    (axis cs:0,0.6) -- cycle;

  \pgfplotsinvokeforeach{-0.80,-0.60,-0.40,-0.20,0,0.20,0.40,0.60,0.80,1.00,1.20,1.40,1.60,1.80,2.00}{%
    \addplot[
      green!70!black,
      line width=0.4pt,
      domain=0:2.2,
      samples=2
    ] {1.3 - x + #1};
  }%
\end{scope}

      \draw[dotted]          (axis cs:0.30,0) -- (axis cs:0.30,1.4);
      \draw[dotted]          (axis cs:0.60,0) -- (axis cs:0.60,1.4);
      \draw[dotted]          (axis cs:0.70,0) -- (axis cs:0.70,1.4);
      \draw[dotted]          (axis cs:2.00,0) -- (axis cs:2.00,1.4);
      \draw[densely dotted]  (axis cs:0,1.3)  -- (axis cs:2.1,1.3);
      \draw[densely dotted]  (axis cs:2.1,0)  -- (axis cs:2.1,1.3);

      \addplot+[no marks, red, solid, line width=0.9pt, samples=2, domain=0:0.30] {-2*x + 0.60}node[midway,sloped,below]{\(\dot v\ge-\be\)};

      \addplot+[no marks, red, solid, line width=0.9pt, samples=2, domain=0:0.70] {x + 0.60}node[midway,sloped,above]{\(\dot v\le\al\)};

      \draw[red,solid,line width=0.9pt](0.7,1.3)--(2.1,1.3)node[midway,above]{\(v\le\vmax\)}--(2.1,0)node[midway,right]{\(t\le q\)}--(0.3,0)node[midway,below]{\(v\ge0\)};

      \addplot+[no marks, red, dotted, line width=0.9pt, samples=2, domain=0.70:2.00] {-x + 2.00};

      \addplot+[no marks, red, dotted, line width=0.9pt, samples=2, domain=0:0.60] {-x + 0.60};

      \draw[blue,line width=0.9pt,->,densely dashdotted](0.01,0.60)--(0.30,0.02)--(2.1,0.02);
      \draw[blue,line width=0.9pt,densely dashdotted](0.15,0.32)--(0.45,0.02);
      \draw[blue,line width=0.9pt,->,densely dashdotted](0.01,0.59)--(0.70,1.28)--(2.1,1.28);
      \draw[blue,line width=0.9pt,densely dashdotted](0.31,0.89)--(1.18,0.02);
      \draw[blue,line width=0.9pt,densely dashdotted](0.8,1.28)--(2.06,0.02);
      \draw[blue,line width=0.9pt,->,densely dashdotted](1.5,1.28)--(2.1,0.68);

    \end{axis}
  \end{tikzpicture}

  \caption{A velocity–time diagram for the Uniform case showing $t_1 \le t_2 < t_3 < t_4 \le q$ for a particular \(v_0\). The green lines denote $\dot v=-\alpha$ isobars. The red boundaries of the tank represent the constraints, while the blue dash-dotted lines illustrate potential optimal trajectories with various volumes \(d\) of water.}
  \label{fig:uniformpressure}
\end{figure}
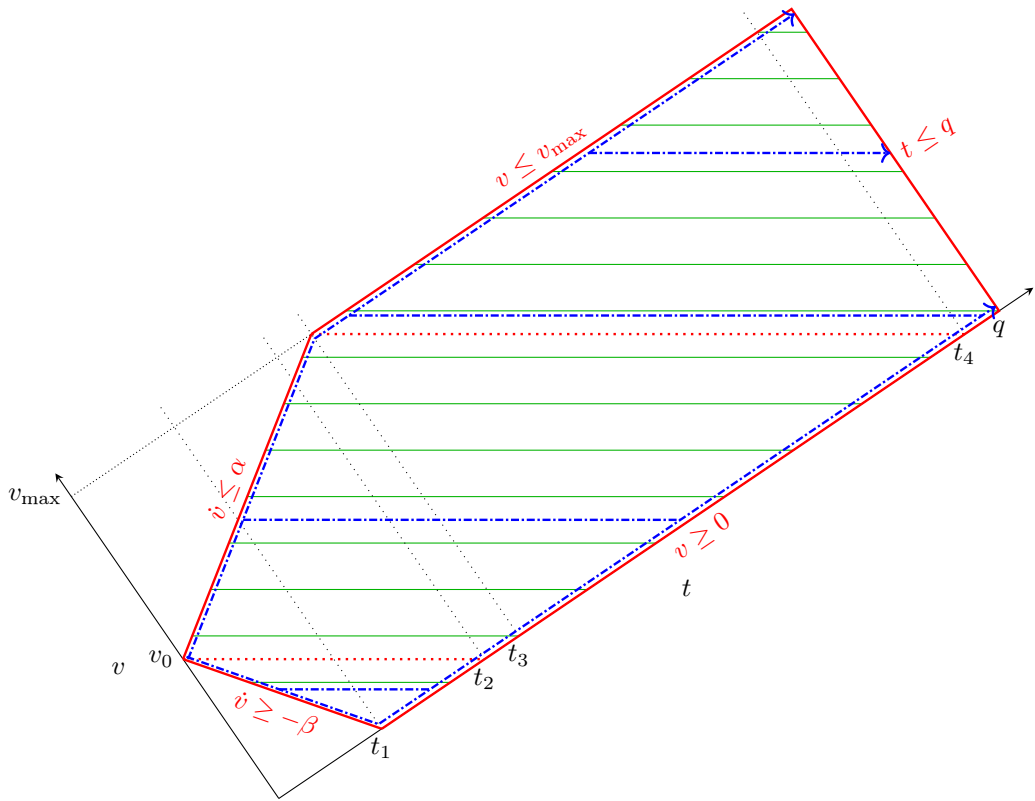
\end{condpics}

\begin{condpics}
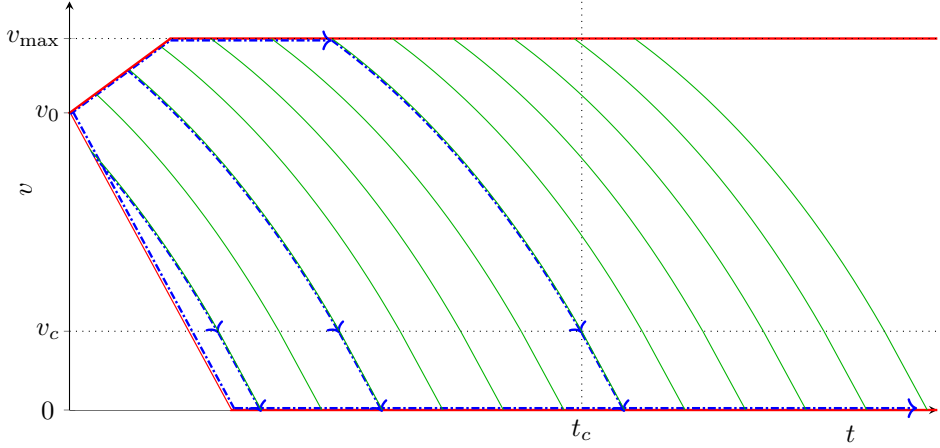
\begin{figure}[ht]
    \centering
    \begin{tikzpicture}
      \pgfmathsetmacro{\vmaxnum}{200}
      \pgfmathsetmacro{\vzero}{160}
      \pgfmathsetmacro{\alpha}{8}
      \pgfmathsetmacro{\lam}{0.07}
      \pgfmathsetmacro{\beta}{20}
      \pgfmathsetmacro{\vc}{42.4}
      \pgfmathsetmacro{\A}{314.29}  

      \pgfmathsetmacro{\taccend}{(\vmaxnum - \vzero)/\alpha}      
      \pgfmathsetmacro{\tzero}{\taccend + 8}                      
      \pgfmathsetmacro{\tc}{\tzero - (1/\lam)*ln((\A - \vmaxnum)/(\A - \vc))}  
      \pgfmathsetmacro{\tend}{\tc + \vc/\beta}                    
      \pgfmathsetmacro{\tblueend}{\vzero/\beta}                   
      \pgfmathsetmacro{\tlin}{\vc/\beta}                          

      \pgfmathsetmacro{\xmax}{ceil(\tend + 15)}
      \pgfmathsetmacro{\xq}{\tend + 12}

      \begin{axis}[
        width=13cm, height=7cm,
        xmin=0, xmax=\xmax,
        ymin=0, ymax=220,
        axis lines=left,
	x label style={at={(axis description cs:0.9,-0.1)},anchor=south},
	y label style={at={(axis description cs:-0.05,0.5)},anchor=west},
        xlabel={$t$}, ylabel={$v$},
        xtick=\empty,
        ytick={0,\vc,\vzero,\vmaxnum},
        yticklabels={$0$,{}, {}, {}},
        clip=false
      ]

        \node[anchor=east] at (axis cs:0,\vzero) {$v_0$};

        \addplot[dotted] coordinates {(\tc,0) (\tc,220)};
        \node[anchor=north] at (axis cs:\tc,0) {$t_c$};


	\addplot[red,line width=0.9pt] coordinates {(\xmax,\vmaxnum)(\taccend,\vmaxnum)(0,\vzero)(\tblueend,0)(\xmax,0)};

        \addplot[blue,line width=0.9pt,densely dashdotted] coordinates {(0,\vzero-1) (\taccend,\vmaxnum-1)};              
        \addplot[blue,line width=0.9pt,->,densely dashdotted, domain=\taccend:\tzero] {\vmaxnum-1};                        
	\addplot[blue,line width=0.9pt,->,densely dashdotted, domain=\tzero:\tc] {\A - (\A - \vc) * exp(-\lam*(\tc - x))-1}; 
        \addplot[blue,line width=0.9pt,->,densely dashdotted, domain=\tc:\tend] {\vc - \beta*(x - \tc)-1};                 

	\addplot[blue,line width=0.9pt,->,densely dashdotted, domain=\tzero-12:\tc-12] {\A - (\A - \vc) * exp(-\lam*(\tc - x-12))-1}; 
        \addplot[blue,line width=0.9pt,->,densely dashdotted, domain=\tc-12:\tend-12] {\vc - \beta*(x - \tc+12)-1};                 

	\addplot[blue,line width=0.9pt,->,densely dashdotted, domain=1:\tc-18] {\A - (\A - \vc) * exp(-\lam*(\tc - x-18))-1}; 
        \addplot[blue,line width=0.9pt,->,densely dashdotted, domain=\tc-18:\tend-18] {\vc - \beta*(x - \tc+18)-1};                 

	\addplot[blue,line width=0.9pt,densely dashdotted]coordinates{(0.2,\vzero)(\tblueend+0.2,0)};
	\addplot[blue,line width=0.9pt,->,densely dashdotted]coordinates{(\tblueend+0.2,1)(\xmax-1,1)};


        \foreach \s in {-18,-15,-12,-9,-6,-3,0,3,6,9,12,15}{
          \addplot[green!70!black, line width=0.4pt,
                   restrict x to domain*={0:\xmax},
                   domain={\tzero+(\s)}:{\tc+(\s)}]
            {\A - (\A - \vc) * exp(-\lam*(\tc - (x - \s)))};

          \addplot[green!70!black, line width=0.4pt,
                   restrict x to domain*={0:\xmax},
                   domain={\tc+(\s)}:{\tc+(\s)+\tlin}]
            {\vc - \beta*(x - (\tc+(\s)))};
        }

        \addplot[draw=none, fill=white] coordinates {(0,\vzero) (\taccend,\vmaxnum) (0,\vmaxnum)} -- cycle;
        \addplot[draw=none, fill=white] coordinates {(0,\vzero) (\tblueend,0) (0,0)} -- cycle;

        \addplot[red, thick] coordinates {(0,\vzero) (\taccend,\vmaxnum)};
        \addplot[dotted, domain=0:\xmax] {\vmaxnum};
        \node[anchor=east] at (axis cs:0,\vmaxnum) {$\vmax$};
	\draw(0,0)--(0,\vmaxnum);
        \node[anchor=east] at (axis cs:0,\vc) {$v_c$};
        \addplot[dotted, domain=0:\xmax] {\vc};

      \end{axis}
    \end{tikzpicture}
    \caption{A velocity–time diagram for the Exponential case. The red boundaries of the tank are as in the Uniform case. The green curves denote the Euler--Lagrange isobars followed by $\beta$ segments. The blue dash-dotted lines illustrate potential optimal trajectories with various volumes \(d\) of water. Note the change from Euler--Lagrange to $\beta$ occurs at velocity $v_c$, which does not depend on $d$. In general, the Euler--Lagrange curve is not tangent to the $-\beta$ segment where they meet.}\label{fig:exppressure}
\end{figure}
\end{condpics}

We rigorously prove all elements of this picture. We start by defining the problem carefully in the language of action integrals in Section \ref{Action}, and proving the existence of an optimiser in Section \ref{sc:exists}. Following a well-placed integration by parts in Section \ref{sc:findit}, we introduce the water tank analogy in Section \ref{pressureinterpretation}. This allows us to prove that the segments of the optimal trajectory must be as seen in the water tank pictures (Theorem \ref{thm:proofofcombinatorics}). The last two sections work out the specific details for the Uniform and Exponential examples. In particular, Theorem \ref{tm:vc} concerns the fairly involved arguments to find the phase transition in the Exponential case from the isobar Euler--Lagrange curves to full \(\be\) deceleration at the critical velocity \(v_c\). The value of \(v_c\) is also determined.

\subsection{Why our model is not fully realistic}
In this paper we ignore your reaction times as a driver. We also ignore the discomfort that you may cause to yourself or your passengers when you adjust your acceleration discontinuously.  A more realistic analysis than the one in this paper would place limits on the jerk (the time derivative of acceleration), and would take into account the dependence of the maximum acceleration and braking rates on velocity. Adding realistic constraints on the jerk would certainly change the optimal trajectories a little in the cases where in our model there is a discontinuity in the acceleration, from $\alpha$ to $0$, from $0$ to a negative value, or from a non-positive value to $-\beta$.  A more realistic optimisation problem could in general only be solved approximately, for example using a discretisation and dynamic programming.
Our main purpose in this paper is to give rigorous derivations of the unique optimal solutions to the slightly less realistic problems that we have posed above in cases \ref{it:uni} and \ref{it:exp}. Nevertheless we expect that the basic qualitative features of the optimal solutions are likely to remain when the model is made more realistic. In particular, we expect the surprising rapid change from Euler-Lagrange deceleration to maximum braking to be a general phenomenon. This is left open for future research.

Another problem that is of interest is to optimise for energy consumption, or for some combination of energy consumption and travel time. There are various possible objective functions, depending on whether we take into account friction losses and regenerative braking, and the possibly nonlinear dependence of these on the velocity and acceleration.  Even in the simple scenario where acceleration is costed in proportion to the increase in kinetic energy $\frac{1}{2}m v^2$, but no energy is recovered during braking, the energy cost even of a simple trajectory which accelerates from $v_0$ to a velocity $v_{\mathrm{top}}$ then brakes to velocity $0$ is $\frac{1}{2}m (v_{\mathrm{top}}^2 - v_0^2)$. To express this in a Lagrangian integral, one would need a discontinuous term $mv\max(0,\dot{v})$. This significantly complicates the rigorous analysis, and we leave it for future work.

\subsection{Analysis of the general case}

Before we specialize to case \ref{it:uni} or case \ref{it:exp}, we will first discuss the more general case \ref{it:renewal}. We can apply methods from classical physics and variational calculus to derive the optimal approach trajectory to a red traffic light that turns green at a random time \(T\).
We parametrize this path by the location \(x(t)\) at time \(t\ge0\). Because of the constraints on braking and acceleration, the function $t \mapsto x(t)$ must be continuously differentiable; in fact $\dot x = \frac{\di}{\di t} x(t)$ is Lipschitz. At the random  moment when the light turns green, the car leaves the path \(t \mapsto x(t)\), accelerates at maximum acceleration \(\al\) to its maximum permitted velocity \(\vmax\) to reach its distant destination \(L\) as quickly as possible. We assume \(T\ge0\) a.s.
Consider a stationary renewal process with i.i.d.~interarrival times $Y_i\geq0$ with finite second moment and cumulative distribution function (CDF) $\Theta: [0, \infty) \to [0,1]$. Let $q = \sup\{t: \Theta(t) < 1\}$. Note that $q = \infty$ if the interarrival times are unbounded. The remaining time until the next arrival in a renewal process is called the residual time. In a stationary renewal process the residual time has probability density given by  (see \cite[\S 3.9]{resnick_adv_sp})
 \begin{equation}
 f(x)=\frac{1}{\mathbb{E}(Y_1)}[1-\Theta(x)].\label{eq:stationary}
 \end{equation}
This shows that any density appearing this way must be bounded and non-increasing on $[0,q)$. By assuming a finite second moment of \(Y_i\), the mean associated with this density becomes finite. 

Motivated by the above, we assume that the density for our random time \(T\) is indeed bounded:
\begin{equation}
    0\leq f(x)\leq K
    \label{eq:boundedf}
\end{equation}
and strictly positive and non-increasing on $[0,q)$. The possible trajectories $x$ will be functions defined on $[0,q)$. 

For example, if we let \(\Theta(x) = \mathbf{1}(x \ge q)\), then \(f\) is the density of the Uniform(\(0,\,q)\) distribution,  recovering Case \ref{it:uni} above. On the other hand, for a memoryless interarrival distribution, where the stationary renewal process is a Poisson Process, equation~\eqref{eq:stationary} yields the density function of the Exponential distribution, i.e. Case \ref{it:exp} above.

We set the initial location \(x(0)=0\), the initial velocity \(\dot x(0)=v_0 \ge 0\) and we assume that the light is at distance \(d\) far enough so that if we brake as hard as possible we will not be forced to pass the traffic light before it turns green. In the case where $q \ge v_0/\beta$, this means that $d\geq\frac{v_0^2}{2\beta}$. In the case where $q < v_0/\beta$, we only need to avoid passing the traffic light before time $q$, so we assume that $d \ge q(v_0 - q\beta/2)$.  We also insist that \(\dot x(t) \ge 0\) at all times \(t\): reversing is not allowed. Since \(x(t)\) is the path followed before the light turns green, \(x(t)\le d\) for all $t \in [0,q)$. Hence the limit \(\lim_{t \to q} x(t)\) exists by the condition $\dot x(t) \ge 0$, and satisfies \(\lim_{t \to q} x(t) \le d\).
Later we will show that we can consider only trajectories such that \(\lim_{t\to q}x(t)=d\).

At the random time $T$ when the traffic light switches from red to green, the car is at location \(x(T)\) and travelling at velocity \(v(T)\). From this moment, accelerating at rate \(\al\) to reach full speed \(\vmax\) and then travelling at full speed to the destination requires time $k(x(T),v(T))$, where
\begin{equation}
k(x,\,v)=\frac{\vmax-v}\al+\Bigl(L-x-\frac{v+\vmax}2\cdot\frac{\vmax-v}\al\Bigr)\cdot\frac1\vmax=\frac1{2\al\vmax}(\vmax-v)^2+\frac1\vmax(L-x).
\label{kdefinition}
\end{equation}
 That is, the absolute arrival time is
 \(T+k\bigl(x(T),\,v(T)\bigr)\).
Next, we take the distribution of $T$ into account to calculate the expected time to reach our destination:
\[
 \Eb\bigl[T+k\bigl(x(T),\,\dot x(T)\bigr)\bigr]=\int_0^q \bigl[t+k\bigl(x(t),\,\dot x(t)\bigr)\bigr]\cdot f(t)\di t.
\]
It is this quantity that we aim to minimise by picking the optimal trajectory \(x(t)\) subject to our constraints.

\subsection{Related work}
Many works consider variants of our problem. Perhaps the closest to our question is a question on MathOverflow raised by P\'alv\"olgyi \cite{mathofl_opt_red_light}. This post poses the problem that we address in this paper  but with no constraint on the maximum braking rate, so effectively $\beta = \infty$ there. Carneiro gives the most important building blocks of the optimal solution in a reply. These are the Euler--Lagrange curves of a certain action integral formulation that we also use shortly in this article. Carneiro also mentions some of the boundary conditions that will be the main objects of our study. We remark that this reply concerns the special cases of Uniform and Exponential distributions, which are also the two special cases that we treat in detail.

The main contribution of our work is to incorporate the constraints on maximum velocity, acceleration and deceleration. This turns out to be far from trivial, requiring us to go well beyond Carneiro's observations. In particular, in the Exponential case, we prove the existence of a critical velocity at which the optimal trajectory suddenly switches to \(-\be\) deceleration from an Euler--Lagrange phase, with a discontinuity in deceleration. When we first encountered this phenomenon it came to us as a surprise. 

There have been several other works published in peer-reviewed journals. Optimisations for energy consumption, which we do not consider here, have been investigated by Katz \cite{katz_how_traffic_light} and Lawitzky, Wollherr, Buss \cite{law_woll_buss_energy_opt_traffic_lights}. The former considers the Uniform distribution and analyses the probability of reaching the green or the red phases of the lights, with obvious implications for the energy consumption. The latter work sets up a modified Bellman equation for the problem and provides algorithms to solve it. Numerical results for various distributions follow. 

Stochastic optimal control problems naturally lend themselves to dynamic programming methods for approximating optimal solutions, see the books of Bertsekas \cite{Bert_dynprogI,Bert_dynprogII}. Here we take a different approach and rigorously derive the explicit form of the optimal trajectory in the Exponential and Uniform cases. We use concepts of action integrals from classical mechanics and intuition from statics of fluids. Connecting optimal control problems to equations of motion in mechanics is not new, see for example Join, Delaleau and Fliess \cite{join_del_fli_e-l_opt-cont}.

\section{Action integral and Euler--Lagrange curves}\label{Action}

Define the \emph{Lagrangian} $\Lc$ by 
\begin{equation}\label{eq:LagrangianDef}
\Lc(x,\,v,\,t)=\bigl[t+k(x,\,v)\bigr]\cdot f(t).
\end{equation}
The expected time to reach the destination becomes a classical action integral:
\begin{equation}
 S:\,=\int_0^q \Lc\bigl(x(t),\,\dot x(t),\,t\bigr)\di t.\label{eq:ai}
\end{equation}

Notice that so far we have not incorporated our constraints \(0\le\dot x(t)\le\vmax\) and \( -\beta \le \ddot x(t)\le\al\) into the function \(k\) or the Lagrangian. We will solve the unconstrained variational problem first, then deal with the constraints later.

We recap the idea behind deriving the \emph{Euler--Lagrange equations of motion}. The \emph{principle of least action} in physics states that mechanical systems follow paths \(t \mapsto x(t)\) that are critical points of an action functional like~\eqref{eq:ai} in the space of paths meeting endpoint constraints. If we forget about our restrictions \(0\le\dot x(t)\le\vmax\) and \(-\beta \le \ddot x(t)\le\al\), then this is exactly what we are looking for in our stochastic control problem. The idea is that perturbing an optimal path by adding a small variation \(\ve\cdot\eta(t)\) can only make the action larger, for any choice of the function \(\eta\). Hence, for any \(\eta\), the modified action
\[
 S_\ve:\,=\int_0^q\Lc_\ve\bigl(x(t),\,\dot x(t),\,t\bigr)\di t
 =\int_0^q \Lc\bigl(x(t)+\ve\eta(t),\,\dot x(t)+\ve\dot\eta(t),\,t\bigr)\di t
\]
is minimal at \(\ve=0\) for any \(\eta\). Here we consider only an interval of time \(0\le a<b<q\) and assume that \(\eta\equiv0\) outside \((a,\,b)\). This will inform us about the unconstrained solution in this interval, for any boundary conditions. We can use this solution in any interval where it actually satisfies our constraints. As explained in \cite{land1}, we arrive at the Euler--Lagrange equation
\[
 \frac{\pt\Lc}{\pt x}-\dd t\frac{\pt\Lc}{\pt v}=0.
\]
Substituting in our Lagrangian, for \(t \in [0,q)\) we obtain
\[
 \begin{aligned}
  f(t)\cdot\frac{\pt k}{\pt x}-\dd t\Bigl(f(t)\cdot\frac{\pt k}{\pt v}\Bigr)&=0\\
  \frac{\pt k}{\pt x}-\frac{\pt k}{\pt v}\cdot\dd t{\ln f(t)}-\dd t\frac{\pt k}{\pt v}&=0\\
  -\frac1{\vmax}+\frac1{\al\vmax}(\vmax-v)\cdot\dd t{\ln f(t)}+\frac1{\al\vmax}\dd t(\vmax-v)&=0
 \end{aligned}
\]
that is,
\[
 \ddot x+\dd t{\ln f(t)}\cdot\dot x-\dd t{\ln f(t)}\cdot\vmax+\al=0.
\]
Writing \(v(t)=\dot x(t)\) we have the ODE for the function \(v=v(t)\)
\begin{equation}
 \dot v+\dd t{\ln f(t)}\cdot v-\dd t{\ln f(t)}\cdot\vmax+\al=0.\label{eq:masterode}
\end{equation}
We will refer to solutions of this ODE as \emph{Euler--Lagrange curves}, often represented by $\vel$.

\section{Existence of an optimal velocity profile}\label{sc:exists}

Before we specialise to the case of any particular distribution, we will show that under the general assumptions stated so far there exists at least one optimal trajectory. First, we need to formalise the space of velocity functions we are working with.

\begin{df}[$(\alpha,\beta)$-Lipschitz Function] For any $q \in (0,\infty) \cup \{\infty\}$, and constants $\alpha,\beta \ge 0$,
a function $f : [0,q) \to \mathbb{R}$ is said to be \emph{$(\alpha,\beta)$-Lipschitz} when
\[
-\beta\,(t_2 - t_1) \;\le\; f(t_2) - f(t_1) \;\le\; \alpha\,(t_2 - t_1)
\quad \forall\; 0 \le t_1 \le t_2 < q.
\]
\end{df}
Every allowed velocity function $t \mapsto v(t)$ for our car is an $(\alpha,\beta)$-Lipschitz function. A combination of this condition, the bounds $0 \le v(t) \le \vmax$, and the total distance constraint gives us two  spaces of functions to work with:
\begin{df}
We define two spaces of functions $V$ and $W$ as follows:
\[
V \;=\;
\Bigl\{\,v:[0,q)\to[0,v_{\max}]\;\Big|\; v(0) = v_0 \text{ and }
v\text{ is $(\alpha,\beta)$-Lipschitz with}\ 
\int_{0}^{q}v(t)\,\mathrm{d}t= d
\Bigr\}.
\]
\[
W \;=\;
\Bigl\{\,v:[0,q)\to[0,v_{\max}]\;\Big|\;v(0) = v_0 \text{ and } 
v\text{ is $(\alpha,\beta)$-Lipschitz with}\ 
\int_{0}^{q}v(t)\,\mathrm{d}t\le d
\Bigr\}.
\]
\end{df}
Notice that $V \subseteq W$. Our assumption that $d$ is large enough that we can avoid passing the traffic light until time $q$ ensures that $V$ and $W$ are not empty. The action integral functional $S$ given in equation~\eqref{eq:ai} may be rewritten in terms of $v$ using $x(t) = \int_0^t v(s) \di s$:
\begin{equation}\label{eq:Sdef}
S(v) = \int_0^q \Lc\left(\int_0^t v(s) \di s, v(t),t\right) \di t.
\end{equation}
We want to show that there exists some $v^* \in V$ such that
\[S(v^*) = \min_{v \in W} S(v).\]
In order to do this we define a metric $D$ on $W$ to make it into a compact metric space on which $S$ is a continuous functional. 
The metric $D$ is defined using $f$ and $F$, the probability density function and cumulative distribution function of $T$, the residual time until the traffic light turns green. As motivated in the introduction by our choice to obtain $T$ via a stationary renewal process, we assume that $f$ is a bounded non-increasing function. We also assume that \[ \mathbb{E}(T)  = \int_0^q 1 - F(t) \di t < \infty,\] which is equivalent to assuming that the interarrival times $Y_i$ have finite second moment.

\begin{df}

Let $\zeta: [0,q) \to (0,q)$ be given by $\zeta(t) = 1 - F(t) + f(t)$. For $g,h \in W$, let
\[D(g,h) = \int_0^q |g(t) - h(t)|\,\zeta(t) \di t.\]
\end{df}
Note that $\zeta$ is non-increasing, and integrable:
\[ \int_0^q \zeta(t) \di t = 1 + \mathbb{E}(T) < \infty.\]
\begin{lm}
 $D$ is a metric on $W$ and $D$-convergence is equivalent to locally uniform convergence on $[0,q)$.
\label{dconv-h}
\end{lm}

\begin{proof}
Since every $g\in W$ is Lipschitz and hence continuous, $D$ is indeed a metric on $W$. (If $D(g,h)=0$ then $g=h$ a.e., hence $g=h$ everywhere by continuity.) 
Define $M:=\max\{\alpha,\beta\}$.

\smallskip
\noindent\emph{(i) $D(g_n,g)\to0$ $\Longrightarrow$ $g_n\to g$ locally uniformly.}
Fix $0 < \kappa < q$ and $\varepsilon>0$. Suppose, towards a contradiction, that
\[
\sup_{t\in[0,\kappa]}|g_n(t)-g(t)|\ \ge\ \varepsilon
\]
for infinitely many $n$. For each such $n$ choose $t_0\in[0,\kappa]$ with $|g_n(t_0)-g(t_0)|\ge\varepsilon$.
By the $(\alpha,\beta)$-Lipschitz property of $g_n$ and $g$,
\[
|g_n(t)-g(t)|\ \ge\ |g_n(t_0)-g(t_0)|-M|t-t_0|-M|t-t_0|
\ \ge\ \varepsilon-2M|t-t_0|.
\]
Let
\[
\delta\;=\;\min\Bigl(\frac{\varepsilon}{2M},\,\frac{\kappa}{2}\Bigr).
\]
Then for all $t\in [t_0-\delta,\,t_0+\delta]\cap[0,\kappa]$ we have $|g_n(t)-g(t)|\ge \varepsilon/2$. Because $\zeta$ is non-increasing and $t\le \kappa$ on this set, $\zeta(t)\ge \zeta(\kappa)$. Hence, writing
$I:=[t_0-\delta,\,t_0+\delta]\cap[0,\kappa]$ and noting $|I|\ge \delta$, we obtain
\[
D(g_n,g)
=\int_{0}^{q}|g_n(t)-g(t)|\,\zeta(t) \di t
\ \ge\ \int_{I}\frac{\varepsilon}{2}\,\zeta(t) \di t
\ \ge\ \frac{\varepsilon}{2}\,\zeta(\kappa)\,|I|
\ \ge\ \frac{\varepsilon}{2}\,\zeta(\kappa)\,\delta
\ >\ 0,
\]
a contradiction to $D(g_n,g)\to0$. Therefore $\sup_{t\in[0,\kappa]}|g_n(t)-g(t)|\to0$. Since $\kappa>0$ was arbitrary, $g_n\to g$ locally uniformly.

\bigskip
\noindent\emph{(ii) $g_n\to g$ locally uniformly $\Longrightarrow$ $D(g_n,g)\to0$.}
Let $\varepsilon>0$. Split the $D$-integral at some large $R>0$:
\[
D(g_n,g)
=\underbrace{\int_{0}^{R}|g_n-g|\,\zeta(t) \di t}_{\text{(A)}}
+\underbrace{\int_{R}^{q}|g_n-g|\,\zeta(t) \di t}_{\text{(B)}}.
\]

\emph{(A) Head estimate.}
By uniform convergence on $[0,R]$, there exists $N$ and $\delta>0$ such that for all $n\ge N$,
$\sup_{t\in[0,R]}|g_n(t)-g(t)|<\delta$. Hence, for $n\ge N$,
\[
\text{(A)}\ \le\ \delta\int_{0}^{R}\zeta(t)\di t.
\]
Since $\int_{0}^{R}\zeta(t)\di t<\infty$, choose $\delta$ small enough that
$\delta\int_{0}^{R}\zeta(t)\di t<\varepsilon/2$.

\emph{(B) Tail estimate.}
Because $0\le g_n,g\le v_{\max}$, we have $|g_n-g|\le 2v_{\max}$, whence
\[
\text{(B)}\ \le\ 2v_{\max}\int_{R}^{\infty}\zeta(t)\di t.
\]
By integrability of $\zeta$, choose $R$ so large that $\displaystyle \int_{R}^{\infty}\zeta(t)\di t<\frac{\varepsilon}{4v_{\max}}$,
yielding $\text{(B)}<\varepsilon/2$.

Combining (A) and (B), for this $R$ and all $n\ge N$ we have $D(g_n,g)<\varepsilon$. Hence $D(g_n,g)\to0$.

\smallskip
This completes the proof of the equivalence between $D$-convergence and locally uniform convergence.
\end{proof}

We now show some properties of the metric space \((W,D)\) and the functional $S$ that will be essential in proving the uniqueness of an optimiser.

\begin{lm}
The metric space \((W,D)\) is compact.
\end{lm}

\begin{proof}
Recall from Lemma~\ref{dconv-h} that the topology induced by the  metric $D$ is the topology of locally uniform convergence on $[0,q)$.   In metric spaces, sequential compactness is equivalent to compactness. We will show that $(W,D)$ is sequentially compact; that is to say, every sequence of functions in $W$ has a subsequence that converges locally uniformly (and hence with respect to $D$) to a limit that lies in $W$. Any sequence $(v_n)_{n=1}^\infty$ of functions in $W$ is uniformly bounded, since they take values in \([0,v_{\max}]\), and equicontinuous by the $(\alpha,\beta)$-Lipschitz condition. Fix any increasing sequence $\kappa_1 < \kappa_2 < \cdots$ with $\kappa_m \to q$, so that the compact intervals $[0,\kappa_m]$ exhaust $[0,q)$ (this is possible whether $q$ is finite or infinite). By the Arzel\`a-Ascoli theorem, there is a subsequence of $(v_n)$ converging uniformly on $[0,\kappa_1]$; passing to a further subsequence, it converges uniformly on $[0,\kappa_2]$; and so on. Taking the diagonal subsequence, we obtain a subsequence converging uniformly on each $[0,\kappa_m]$, and hence, since every compact subset of $[0,q)$ is contained in some $[0,\kappa_m]$, converging locally uniformly on $[0,q)$ to some function \(v : [0,q) \to \mathbb{R}\). We must show that $v \in W$. The pointwise convergence implies that $v$ takes its values in the closed interval $[0,\vmax]$, and that $v(0) = \lim_{n \to \infty} v_n(0) = v_0$. It also implies that for any $0 \le t_1 \le t_2 < q$ we have
\[
              v(t_{2}) - v(t_{1})
              =   \lim_{n \to \infty} (v_n(t_2) - v_n(t_1))
              \le \,\alpha(t_{2}-t_{1}),
            \]
and
            \[
              v(t_{2}) - v(t_{1})
              =   \lim_{n \to \infty} (v_n(t_2) - v_n(t_1))
              \ge \,-\beta(t_{2}-t_{1}).
            \]
That is, \(v\) is \((\alpha,\beta)\)-Lipschitz. Finally, by Fatou's lemma, 
\[ \int_0^q v(t) \di t  = \int_0^q \lim_{n \to \infty} v_n(t) \di t \le \liminf_{n \to \infty} \int_0^q v_n(t) \di t \le d,\]
since all the functions $v_n$ are non-negative measurable functions.  Therefore \(v\in W\) as required.
\end{proof}

\begin{lm}
\(S\) is continuous on \((W,D)\).  
\end{lm}
\begin{proof}
The function $k(x,v)$ displayed in~\eqref{kdefinition} is a polynomial in $x$ and $v$, so it is bounded and Lipschitz on $[0,d]\times[0,\vmax]$. That is to say, there exists a finite constant $C$ such that for any two pairs of points $(x,g)$ and $(y,h)$ in $[0,d] \times [0,\vmax]$,
\[|k(x,g)-k(y,h)|\leq C(|x-y|+|g-h|).\]
Let $x_g(t) = \int_0^t g(s) \di s$ and $x_h(t) = \int_0^t h(s) \di s$, and note that 
\[ |x_g(t) - x_h(t)| \le \int_0^t |g(s) - h(s)| \di s,\]
so
\[
\begin{aligned}\int_0^q |x_g(t) - x_h(t)| f(t) \di t & \leq  \int_0^q \int_0^t |g(s) - h(s)|\di s\, f(t) \di t\\
& =  \int_0^q |g(s) - h(s)| \int_s^q f(t) \di t \di s\\
& =  \int_0^q |g(s) - h(s)
| (1-F(s)) \di s .
\end{aligned}
\]
Hence
\[
\begin{aligned}
|S(g)-S(h)|
&\leq \int_0^{q} |k(x_g(t),g(t)) - k(x_h(t),h(t))|\, f(t)\, \di t \\
&\leq \int_0^{q} C\big(|x_g(t)-x_h(t)| + |g(t)-h(t)|\big)\, f(t)\, \di t \\
& \le C\int_0^q |g(s) - h(s)| (1-F(s) + f(s)) \di s   \\
&= C\cdot D(g,h).
\end{aligned}
\]
Thus \(S\) is Lipschitz continuous (in particular continuous) on \((W,D)\).
\end{proof}

We are now prepared to prove the existence an optimal velocity trajectory in $V$.
\begin{tm}
\label{thm:existence-minimiser}
For the functional $S: W \to \mathbb{R}$ defined by~\eqref{eq:LagrangianDef} and~\eqref{eq:Sdef},
 there exists at least one \(v^*\in W\) that minimises \(S\), and in fact any minimiser satisfies
\(\displaystyle\int_{0}^{q}v^*(t)\di t=d\),
hence lies in $V$.
\end{tm}

\begin{proof}
Since \((W,D)\) is compact and $S(v)$ is continuous, it follows from the Extreme Value Theorem that there exists at least one minimiser of $S$ in $W$. We will show that every minimiser in fact lies in $V$. Suppose (towards a contradiction) that \(v^*\in W\) minimises $S$ but \(\displaystyle\int_{0}^{q}v^*(t)\di t = d_{0}<d\).  We construct
\(w\in W\) with \(\int_0^q w(t) \di t <d\), \(w\ge v^*\), and \(w\not\equiv v^*\) and then show that this contradicts the minimality of $v$.

\begin{enumerate}[label=(\alph*)]

  \item \emph{Constructing $w$ in the case $q = \infty$.}
    For a small \(\varepsilon>0\), set
    \[
      w(t)=\max\bigl(v^*(t),\,v^*(t(1-\varepsilon))\bigr),
      \quad
      y(t)=\min\bigl(v^*(t),\,v^*(t(1-\varepsilon))\bigr).
    \]
Note that    
 \(w(t)\ge v^*(t)\) and \(0\le w(t)\le v_{\max}\) for every $t \ge 0$, and since
 \(w\) is the pointwise maximum of two $(\alpha,\beta)$-Lipschitz functions, it is itself $(\alpha,\beta)$-Lipschitz. By the change of variables \(s=t(1-\varepsilon)\),
    \[
      \int_{0}^{\infty}v^*(t(1-\varepsilon))\di t
      =\frac{1}{1-\varepsilon}\int_{0}^{\infty}v^*(s)\di s
      =\frac{d_{0}}{1-\varepsilon}.
    \]
    For fixed $t > 0$, as $\varepsilon \to 0$, we have $t(1-\varepsilon) \to t$, so by continuity of $v^*$, $v^*(t(1-\varepsilon)) \to v^*(t)$. Hence $y(t) \to v^*(t)$ pointwise. Since $0 \leq y(t) \leq v^*(t)$ and $\int_0^\infty v^*(t)\di t = d_0 < \infty$, the dominated convergence theorem yields $\int_0^\infty y(t)\di t \to d_0$. Hence
    \[
      \int_{0}^{\infty}w(t)\di t
      =\int_0^\infty v^*(t) \di t + \int_0^\infty v^*(t(1-\varepsilon))\di t \;-\;\int_0^\infty y(t)\di t
      =\Bigl(1+\tfrac{1}{1-\varepsilon}\Bigr)d_{0}
        \;-\;\int_0^\infty y(t) \di t      \;\xrightarrow[\varepsilon\to0]{}\;d_{0}.
    \]
    Thus for \(\varepsilon\) sufficiently small, \(\int_0^\infty w(t) \di t <d\), so \(w\in W\).
If there exists no $t' > 0$ such that $w(t') > v^*(t')$ then we have $v^*(t) \ge v^*((1-\varepsilon)t)$ for all $t > 0$, but this implies \[\int_0^\infty v^*(t) \di t \ge \int_0^\infty v^*((1-\varepsilon)t) \di(t)  = \frac{1}{1-\varepsilon}\int_0^\infty v^*(t) \di t\,  \] which is only possible if $v^* \equiv 0$, since we know $0 \le \int_0^\infty v^*(t) \di t \le d$ and $v^*(t) \ge 0$ for all $t$. However, we know $v^* \not\equiv 0$ because that is the unique \emph{maximiser} of $S$ and since $d > 0$, $W$ contains at least one non-zero trajectory that is better. Hence there exists at least one $t' > 0$ such that $w(t') > v^*(t')$, and since $v$ and $w$ are continuous there is a non-empty open set of such $t'$.
\item\emph{Constructing $w$ in the case $q < \infty$.} Let $v^+$ be trajectory starting at $v^+(0) = v_0$ which accelerates at rate $\alpha$ until it reaches velocity $\vmax$ and then remains at $\vmax$ until time $q$. For a small $\varepsilon > 0$, define
\[w(t) = (1-\varepsilon) v^*(t) + \varepsilon v^+(t).\]
Then $w(t) \ge v^*(t)$ and $0 \le w(t) \le \vmax$ for all $t \in [0,q]$. Moreover, $w$ is $(\alpha,\beta)$-Lipschitz since it is a convex linear combination of $(\alpha,\beta)$-Lipschitz functions. We have
\[\int_0^q w(t) \di t  = (1-\varepsilon)\int_0^q v^*(t) \di t  + \varepsilon \int_0^q v^+(t) \di t\]
so for sufficiently small $\varepsilon$ we have $\int_0^q w(t) \di t < d$ and hence $w \in W$. Because of our assumption that the traffic light is near enough to make our problem non-trivial, (see equation~\eqref{eq:trafficlightisnear}), $v^{+} \not\in W$, so  $v^* \not\equiv v^{+}$.  Hence there exists a non-empty open set of times $t' \in [0,q)$ such that $w(t') > v^*(t')$. 
  \item \emph{Contradicting statement.}  
At every time $t \in [0,q)$, the car with velocity function $w$ is both at least as fast as the car with velocity function $v$, and at least as far forward; for a non-empty open set $U$ of times \(t'\) when \(w(t')>v^*(t')\),  it is strictly further forward, and there is a positive probability that $T \in U$. Hence a car following trajectory $w$ has strictly smaller expected arrival time than a car following trajectory $v$, contradicting the minimality of $v$.
    
\end{enumerate}

Therefore \(\displaystyle\int_{0}^{q}v^*(t)\di t = d\), and \(v^*\in V\).

\end{proof}
Now that we have have shown that there is at least one optimal trajectory $v^{*}$ in the general case of a bounded non-increasing probability density function $f$, we can examine the structure of optimal trajectories.

\section{Finding an Optimal Trajectory}\label{sc:findit}

\noindent
We now seek to minimise the action integral introduced in Section~\ref{Action}. Recall that the action is
\[
\int_0^q\Bigl(t+\frac1{2\al\vmax}(\vmax-\dot x)^2+\frac1\vmax(L-x)\Bigr)\cdot f\di t
\]
subject to $x(t) = \int_0^t v(s) \di s$ for $v \in W$, i.e.~$\dot{x} = v$.  The Lagrangian
can be reduced by removing constant terms $(\frac{L}{\vmax}+\frac{\vmax}{2\al})\cdot f(t)$ and $t \cdot f(t)$, since they do not affect minimisation of the integral.  We also multiply the action integral by the constant $2\alpha \vmax$, so that our goal is now to minimise
\[
 \int_0^q\Bigl(\dot x^2-2\vmax\dot x-2\al x\Bigr)\cdot f \di t.
\]
We have already seen in Theorem~\ref{thm:existence-minimiser} that every minimiser $v^*$ lies in $V$, so we now assume that $\int_0^q \dot{x}(s) \di s = d$.   Hence we may choose an arbitrary constant $B$ (effectively a Lagrange multiplier) and instead minimise
\[
 \int_0^q\Bigl(\dot x^2-2\vmax\dot x-2\al x\Bigr)\cdot f+2B\dot x\di t.
\]
Notice that  with the cumulative density function \(F\) for the density \(f(t)\),
\[
 \int_0^qxf\di t=x(q)F(q)-x(0)F(0)-\int_0^q\dot xF\di t=d-\int_0^q\dot xF\di t.
\]
The constant \(d\) is irrelevant to the minimisation, so we remove it. Finally, our goal is to minimise the following integral:
\begin{equation}\label{eq:S-hat}
 \widehat{S} = \int_0^qf v^2+2(\al F+B-\vmax f) v \di t.
\end{equation}

\subsection{Pressure interpretation}\label{pressureinterpretation}

Next, we define a new dummy variable $C$ which allows us to write
\(
f v^2+2(\al F+B-\vmax f) v
\)
as a definite integral. Define the \emph{pressure} by
\begin{equation}
P_B(t,\,C)=-(\vmax - C)f(t)+\al F(t)+B\label{eq:pdef}.
\end{equation}
Then the integrand in~\eqref{eq:S-hat} is equal to $2\int_0^v P_B(t,C)  \di{C}$. Hence our problem is to minimise the double integral
\begin{equation}
 S(v) := \int_0^q\int_0^{v}P_B(t,\,C)\di C\di t.
\label{eq:swithpressure}
\end{equation}
Here we have divided by a factor of 2 but this does not change the minimisation. From now on we work with the specific action $S$ given in equation~\eqref{eq:swithpressure}, rather than the abstract one in equation~\eqref{eq:ai}.

Note that the pressure is bounded by 
\begin{equation}
    -\vmax K + B\leq P_B(t,\,C) \leq \al + B
    \label{eq:boundedp}
\end{equation}
since $0\leq f \leq K$ by \eqref{eq:boundedf} and $0\leq C\leq \vmax$ on the domain $t\in[0,q)$. A small change $\di C$ in the variable \(C\) is an infinitesimal unit of height and a change $\di t$ is a unit of width. Hence, $\di C \di t$ can be thought of as an infinitesimal element of two-dimensional \emph{volume} of an imaginary liquid. The distance constraint imposed by a Lagrange multiplier corresponds to having a fixed volume of liquid.

It makes sense to call $P_B(t,C)$ pressure because (potential) energy is the volume-integral of pressure. Moreover, we can now treat the Euler--Lagrange curves as natural isobar level lines of a filled container. The constrained parts of any optimal solution curve (phases where the acceleration is $\alpha$, the deceleration is $\beta$, or the velocity is $0$ or $\vmax$) can be thought of as rigid plates that are pushed by the force acting from the pressurised liquid within the tank. Unlike a regular tank of liquid, our liquid fills according to level lines prescribed by Euler--Lagrange curves given by
\begin{equation}
\vel(t)= \vmax-\frac B{f(t)}-\al\frac{F(t)}{f(t)}.
\label{eulerlagrangegeneral}
\end{equation} 
By varying the arbitrary constant $B$ we obtain a family of Euler--Lagrange curves.
The reader can verify that these trajectories are indeed Euler--Lagrange curves by substituting $\vel$ into \eqref{eq:masterode}.

The Euler--Lagrange curves \eqref{eulerlagrangegeneral} are simply the isobar curves of the pressure \(P_B(t,C)\), and it is clear why this is so. Suppose the point \(\bigl(t_0,v(t_0)\bigr)\) is on an isobar curve \(v(t)\) of total volume \(d\). We can set \(B\) so that \(P_B(t_0,v(t_0))=0\), which also means that the pressure is zero along this isobar. Due to the volume constraint, any other curve must go both above and below \(v(t)\) at different regions of \(t\). As \(P_B(t,C)\) is monotone in \(C\), the regions above \(v(t)\) contribute with positive values in the action integral, while the regions below, which are missing liquid up to the isobar, are missing negative contributions. Hence this other curve would be better off equalising between these two regions so it cannot be optimal.

\begin{lm}\label{lem:movetowardsEL}
 Suppose $v, v'$ are distinct elements of $V$ and $\vel$ is some Euler--Lagrange curve, as in~\eqref{eulerlagrangegeneral}. Suppose that for every $t \in [0,q)$, $v'(t)$ lies in the closed interval with endpoints $v(t)$ and $\vel(t)$. Then $S(v') < S(v)$, so $v$ is not an optimal trajectory.
 \label{interpolatedv}
\end{lm}
\begin{proof}
Let $\mathcal{I} = \{t \in [0,q)\,:\, v(t) < v'(t) \le\vel(t)\}$ and $\mathcal{J} = \{t \in [0,q)\,:\, v(t) > v'(t) \ge\vel(t)\}$. Then $v(t) = v'(t)$ for all $t$ outside $\mathcal{I} \cup \mathcal{J}$, and due to the integral constraints $\int_0^qv(t) \di t = d = \int_0^qv'(t) \di t$ we have
\[\int_\Ic v'(t) - v(t) \di t = \int_\Jc v(t) - v'(t) \di t.\]
Let $p$ be the constant value of the pressure  $P_B(t,\vel(t))$. Then for $t \in \Ic$ and $v(t) \le C \le v'(t)$ we have $P_B(t,C) < p$, and for $t \in \Jc$ and $v'(t) \le C \le v(t)$ we have $P_B(t,C) > p$, thanks to the monotonicity of $P_B(t,C)$ with respect to $C$.
Now 
\[
 \begin{aligned}
  S(v') - S(v)&= \int_{I} \int_{v(t)}^{v'(t)} P_B(t,C) \di C \di t - \int_{J} \int_{v'(t)}^{v(t)} P_B(t,C) \di C \di t\\
  &<  \int_{I} \int_{v(t)}^{v'(t)} p \di C \di t - \int_{J} \int_{v'(t)}^{v(t)} p \di C \di t < 0.\qquad\qquad\qquad\qed
 \end{aligned}
\]\renewcommand{\qedsymbol}{}
\end{proof}

\subsection{Elements of the optimal trajectory}

As we will see later, the acceleration and velocity constraints can move the solution away from Euler--Lagrange, but the liquid analogy can guide our intuition in the process. Obviously, \(\vmax\) restricts the height, while the acceleration bounds restrict the slope of the liquid's surface. These latter can be thought of as slanted, but possibly moveable walls of the pool. As we now show, the optimal trajectory can only be formed from a combination of segments of Euler--Lagrange curve, $\alpha$ acceleration, $\beta$ deceleration and constant $\vmax$ or 0 velocity. 

From now on we will make the following assumption:
\begin{assn}\label{ass:sign-condition}
For any given $B$, the expression
\[
  \dot{\vel}(t) + \beta = \frac{\dot f(t)\cdot (B + F(t))}{f(t)^2} - \alpha + \beta
\]
satisfies exactly one of the following:
\begin{enumerate}
  \item it is positive for all $0 \le t < q$;
  \item it is zero for all $0 \le t < q$;
  \item it is negative for all $0 \le t < q$;
  \item it changes sign exactly once, from positive to negative, on the interval $(0,q)$.
\end{enumerate}
\end{assn}
This will be relevant for examples where the slope of \(\vel\) can become steeper than \(-\be\), and remains steeper permanently, such as the Exponential case. However, if the slope later returns above \(-\be\), then this causes complications that go beyond the scope of this paper.

There exist distributions $F$ for which some Euler--Lagrange curves fail to be concave, and then there exists a choice of $\beta$ such that the assumption above does not hold. One such distribution has cumulative distribution function

\[
F(t) = c \left( 
  2t 
  - \frac{1}{p} 
    \ln\!\left(
        \dfrac{\cosh\!\big(p(t - 1)\big)}{\cosh(-p)} \cdot 
        \dfrac{\cosh\!\big(p(t - 2)\big)}{\cosh(-2p)}
    \right)
\right)
\]
for positive constants $p$ and $c$, where $c$ is chosen so that $\lim_{t \to \infty} F(t) = 1$.

\begin{condpics}
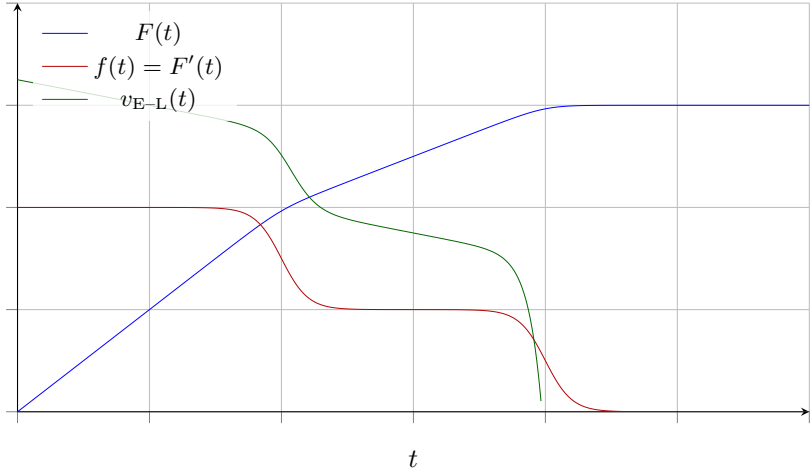
\begin{figure}[h!]
\centering
\begin{tikzpicture}
  \begin{axis}[
    width=12cm, height=7cm,
    domain=0:3, samples=600,
    xmin=0, xmax=3,
    ymin=0, ymax=8,
    axis lines=left,
    xlabel={$t$}, ylabel={}, 
    grid=both,
    tick align=outside,
    xticklabels={}, 
    yticklabels={}, 
    legend style={
      draw=none,
      at={(0.02,0.98)},
      anchor=north west,
      fill=white,
      fill opacity=0.8,
      text opacity=1,
      font=\small
    },
    every axis plot/.append style={very thin, no markers}
  ]

    \addplot[smooth, blue, restrict y to domain=0:8]
      ({x},
       {2*x - (1/10)*ln( (cosh(10*(x-1))/cosh(10)) * (cosh(10*(x-2))/cosh(20)) )}
      );
    \addlegendentry{$F(t)$}

    \addplot[smooth, red!70!black, restrict y to domain=0:8]
      ({x},
       {2 - tanh(10*(x-1)) - tanh(10*(x-2))}
      );
    \addlegendentry{$f(t)=F'(t)$}

    \addplot[smooth, green!40!black, restrict y to domain=0:8]
      ({x},
       {7 - 2/(2 - tanh(10*(x-1)) - tanh(10*(x-2)))
        - (2*x - (1/10)*ln( (cosh(10*(x-1))/cosh(10)) * (cosh(10*(x-2))/cosh(20)) ))
          /(2 - tanh(10*(x-1)) - tanh(10*(x-2)))}
      );
    \addlegendentry{$\vel(t)$}

  \end{axis}
\end{tikzpicture}
\caption{Plot of $F(t)$, its derivative $f(t)=F'(t)$, and the associated Euler--Lagrange curve $\vel(t)$.}
\end{figure}
\end{condpics}

\begin{tm}
\label{thm:proofofcombinatorics}
Assuming \ref{ass:sign-condition}, any optimal trajectory $v^{*} \in V$ must be constituted solely of some combination of finitely many segments of Euler--Lagrange curve, $\alpha$ acceleration, $\beta$ deceleration and constant $\vmax$ or $0$ velocity. In fact, the phases of positive velocity must occur in one of the following seventeen orders. The boxed cases are those that cannot occur if the Euler--Lagrange curves never have slope less than $-\beta$ while they are positive.
\begin{itemize}
\item $\alpha$, \fbox{$\alpha \leadsto \beta$}, \fbox{$\alpha \leadsto EL \leadsto \beta$},
\fbox{$\alpha \leadsto \vmax \leadsto EL \leadsto \beta$},
$\alpha \leadsto EL$,
$\alpha \leadsto \vmax \leadsto EL$, 
\fbox{$\alpha \leadsto \vmax \leadsto \beta$}, \hbox{$\alpha \leadsto \vmax$},
\item $\beta$, \fbox{$\beta \leadsto EL \leadsto \beta$}, $\beta \leadsto EL$,
\item \fbox{$v_{\max} \leadsto EL \leadsto \beta$}, $v_{\max} \leadsto EL$, \fbox{$v_{\max}\leadsto\beta$}, $\vmax$,
\item \fbox{$EL \leadsto \beta$}, $EL$.
\end{itemize}
 When $q=\infty$,  the cases $\alpha$, $\alpha\leadsto\vmax$ and $\vmax$ are excluded in order to satisfy the fixed volume condition $d$, and every optimal trajectory is followed by a phase of $0$ velocity after it comes to a halt at the traffic light. If $q < \infty$, a final phase of $0$ velocity may or not be present, and if present it cannot follow an $\alpha$ or $\vmax$ phase.
\end{tm}
We will prove this result over the course of this subsection. We approach the problem by eliminating non-optimal curves until we are left with only those satisfying the statement of the above theorem.

Theorem~\ref{thm:proofofcombinatorics} does not completely determine a unique optimal trajectory for given (generic) values of $v_0$ and $d$, since it does not specify the times at which the transitions between the phases occur. For some of the patterns, these transition times are always determined by the initial velocity and the distance constraint $\int_0^q v(t) \di t = d$. Omitting the final $0$ phase if it present, these ten patterns are:
\[ \alpha \leadsto \beta,\; \alpha \leadsto EL,\; \alpha \leadsto v_{\max} \leadsto EL,\; \alpha \leadsto v_{\max} \leadsto \beta,\; \beta,\; \beta \leadsto EL,\; v_{\max} \leadsto EL,\; v_{\max} \leadsto \beta,\; EL.\]
Among these are some patterns that can only occur when either $v_0 = v_{\max}$ or an equation relating between $v_0$ and $d$ is satisfied. For example, the pattern
$\beta \leadsto 0$ requires $d = v_0^2/2\beta$, and the pattern $EL \leadsto 0$ requires a more complicated equation to be satisfied (since there is a unique Euler--Lagrange curve that has velocity $v_0$ at time $0$, and its integral from time $0$ up to the time at which it reaches velocity $0$ must equal $d$). Each of the other patterns can only occur when $v_0$ and $d$ satisfy some (pattern-dependent) inequalities. Phase diagrams which show which patterns are optimal for different regions in $(v_0,d)$-space are discussed in detail in Section~\ref{sec:unifphasediagrams} and Section~\ref{sec:expphasediagrams}.

For the remaining four patterns, those that include a phase of Euler--Lagrange motion that is not the first phase and is followed by a phase of $\beta$ deceleration, we will need to do further work (in Section~\ref{sec:whentoslamthebrakes}) in order to determine the optimal transition times. Again omitting the trailing $0$ velocity phase, these patterns are:
\[ \alpha \leadsto EL \leadsto \beta,\; \alpha \leadsto \vmax \leadsto EL \leadsto \beta,\; \beta \leadsto EL \leadsto \beta,\; \vmax \leadsto EL \leadsto \beta\,\; EL, \leadsto \beta.\]

The overall idea of the proof of Theorem~\ref{thm:proofofcombinatorics} is to analyse the trajectories in $V$ for which it is not possible to use Lemma~\ref{lem:movetowardsEL} to show non-optimality. In applying that lemma, we have the freedom to choose any Euler--Lagrange curve. A key step in our proof is Proposition~\ref{prop:upperandlower}, which says that if none of the possible Euler--Lagrange curves allow us to make an integral-preserving perturbation of $v \in V$ as in the statement of Lemma~\ref{lem:movetowardsEL}, then there must exist a particular value $p_0$ of the pressure for which $v$ has a highly constrained structure on the set $\{t \ge 0: P_B(t,v(t)) > p_0\}$, and a different highly constrained structure on the set $\{t: P_B(t,v(t)) < p_0\}$. We finish by using this result to make a case-by-case analysis of the possible overall shapes of any optimal trajectory.

\begin{df}
 A trajectory $v \in V$ is \emph{locally non-increasable} at a point $t \in (0,q)$ if there exists  $\varepsilon>0$ such that on each of the intervals $(t-\varepsilon,t]$ and $[t,t+\varepsilon)$, $v$ is either a linear function with slope either $\alpha$ or $-\beta$ or a constant function with value $\vmax$, and moreover the restriction of $v$ to $(t-\varepsilon,t+\varepsilon)$ is concave.  Otherwise we say $v$ is \emph{locally increasable} at $t$.  If $U \subseteq (0,\infty)$ is an open set, we say $v$ is locally non-increasable on $U$ when for every $t \in U$, $v$ is locally non-increasable at $t$.
\end{df}

\begin{df}
 A trajectory $v \in V$ is \emph{locally non-decreasable} at a point $t \in (0,q)$ if there exists  $\varepsilon>0$ such that on each of the intervals $(t-\varepsilon,t]$ and $[t,t+\varepsilon)$, $v$ is either a linear function with slope either $\alpha$ or $-\beta$ or a constant function with value $0$, and moreover the restriction of $v$ to $(t-\varepsilon, t+\varepsilon)$ is convex.  Otherwise we say $v$ is \emph{locally decreasable} at $t$.  If $U \subseteq (0,\infty)$ is an open set, we say $v$ is locally non-decreasable on $U$ when for every $t \in U$, $v$ is locally non-decreasable at $t$.
\end{df}

\begin{df}
For an interval $\Ic = [i_l,i_r]$ where $0 < i_l < i_r <q$, and a continuous function $v: \Ic \to [0,\vmax]$, we say $v$ is a \emph{positive tent} on $\Ic$ if and only if there exist $i_1$ and $i_2$ such that $i_l \le i_1 \le i_2 \le i_r$ and 
\begin{itemize}
\item on $(i_l,i_1)$, $v$ has slope $\alpha$,
\item on $(i_1,i_2)$, $v$ has slope $0$ and value $\vmax$, and
\item on $(i_2,i_r)$, $v$ has slope $-\beta$.
\end{itemize}
\end{df}

\begin{lm}
For any $0 < i_l < i_r <q$ and values $v_l, v_r \in [0, \vmax]$ such that $-\beta \le \frac{v_r - v_l}{i_r - i_l} \le \alpha$, there exists a unique positive tent $v: [i_l,i_r] \to [0,\vmax]$ such that $v(i_l) = v_l$ and $v(i_r) = v_r$.
\label{lem:uniquetent}
\end{lm}
\begin{proof}
$v$ is the pointwise supremum of the family of $(\alpha,\beta)$-Lipschitz functions $f: [i_l, i_r] \to [0,\vmax]$ such that $f(i_l) = v_l$ and $f(i_r) = v_r$. This family is non-empty because the affine linear function satisfying these boundary conditions is $(\alpha,\beta)$-Lipschitz. 
\end{proof}

\newcommand{\diagramScale}{0.34}
\newcommand{\SmallTextHeight}{0.6cm}

\newlength{\RowH}
\setlength{\RowH}{2.6cm}  
\colorlet{vmaxcolor}{green!40!black} 
\begin{condpics}
\begin{figure}
\begin{table}[H]
\centering
\begingroup
\setlength{\tabcolsep}{2pt}
\renewcommand{\arraystretch}{0.85}

\begin{minipage}{0.82\textwidth}
\centering
\begin{tabular}{@{}M{0.48\linewidth} M{0.48\linewidth}@{}}

\begin{minipage}[c][\RowH][c]{\linewidth}\centering
\begin{tikzpicture}[baseline=(current bounding box.center),
                    scale=\diagramScale, x=1cm, y=1cm, >=Stealth, line cap=round,
                    every node/.style={font=\scriptsize}]
  \def\il{1.6}\def\ir{8.2}\def\xmax{9.6}
  \pgfmathsetmacro{\delta}{0.25}
  \pgfmathsetmacro{\tentlift}{0.8}
  \pgfmathdeclarefunction{f}{1}{\pgfmathparse{0.9 + 0.22*(#1) + 0.6*(1 - exp(-0.5*(#1)))}}

  \draw[->] (0,0) -- (\xmax,0);
  \draw (\il,0) -- ++(0,0.06) node[above] {$i_l$};
  \draw (\ir,0) -- ++(0,0.06) node[above] {$i_r$};

  \draw[very thick, smooth] plot[domain=0:\xmax, samples=200] (\x,{f(\x)});
  \pgfmathsetmacro{\vli}{f(\il)}
  \pgfmathsetmacro{\vri}{f(\ir)}
  \fill (\il,\vli) circle (0.8pt) node[above left=2pt] {$v(i_l)$};
  \fill (\ir,\vri) circle (0.8pt) node[above right=2pt] {$v(i_r)$};

  \pgfmathsetmacro{\ym}{\vri + \tentlift}
  \pgfmathsetmacro{\cOne}{\ym - \vri}
  \pgfmathsetmacro{\cTwo}{(1+\delta)*(\ym - \vli)}
  \pgfmathsetmacro{\midx}{(\cTwo*\ir + \cOne*\il)/(\cOne + \cTwo)}

  \draw[very thick, red]  (\il,\vli) -- node[sloped, above, pos=0.55] {$\alpha$} (\midx,\ym);
  \draw[very thick, blue] (\midx,\ym) -- node[sloped, above, pos=0.45] {$\beta$}  (\ir,\vri);
\end{tikzpicture}
\end{minipage}
&
\fbox{\parbox[c][\RowH][c]{0.96\linewidth}{%
  \vspace{\SmallTextHeight}\scriptsize \textit{Case 1 (No interference from boundary conditions):
  In this case the tent consists of just $\alpha$  acceleration followed by $\beta$ deceleration.}\vspace{\SmallTextHeight}%
}}
\\[-2pt]

\begin{minipage}[c][\RowH][c]{\linewidth}\centering
\begin{tikzpicture}[baseline=(current bounding box.center),
                    scale=\diagramScale, x=1cm, y=1cm, >=Stealth, line cap=round,
                    every node/.style={font=\scriptsize}]
  \def\il{1.6}\def\ir{8.2}\def\xmax{9.6}
  \pgfmathsetmacro{\fracpos}{0.45}
  \pgfmathdeclarefunction{f}{1}{\pgfmathparse{0.9 + 0.22*(#1) + 0.6*(1 - exp(-0.5*(#1)))}}

  \draw[->] (0,0) -- (\xmax,0);
  \draw (\il,0) -- ++(0,0.06) node[above] {$i_l$};
  \draw (\ir,0) -- ++(0,0.06) node[above] {$i_r$};

  \draw[very thick, smooth] plot[domain=0:\xmax, samples=200] (\x,{f(\x)});
  \pgfmathsetmacro{\vli}{f(\il)}
  \pgfmathsetmacro{\vri}{f(\ir)}
  \fill (\il,\vli) circle (0.8pt) node[above left=2pt] {$v(i_l)$};
  \fill (\ir,\vri) circle (0.8pt) node[above right=2pt] {$v(i_r)$};

  \pgfmathsetmacro{\xmeet}{\il + \fracpos*(\ir - \il)}
  \draw[very thick, red]      (\il,\vli) -- node[sloped, above, pos=0.6] {$\alpha$} (\xmeet,\vri);
  \draw[very thick, vmaxcolor] (\xmeet,\vri) -- node[above, pos=0.55] {$v_{\max}$} (\ir,\vri);
\end{tikzpicture}
\end{minipage}
&
\fbox{\parbox[c][\RowH][c]{0.96\linewidth}{%
  \vspace{\SmallTextHeight}\scriptsize \textit{Case 2 ($\vmax$ on the right): If $v(i_r)=\vmax$ then the tent is capped by a horizontal line of $\vmax$ and has no $\beta$ component.}\vspace{\SmallTextHeight}%
}}
\\[-2pt]

\begin{minipage}[c][\RowH][c]{\linewidth}\centering
\begin{tikzpicture}[baseline=(current bounding box.center),
                    scale=\diagramScale, x=1cm, y=1cm, >=Stealth, line cap=round,
                    every node/.style={font=\scriptsize}]
  \def\il{1.6}\def\ir{8.2}\def\xmax{9.6}
  \pgfmathsetmacro{\fracpos}{0.62}
  \pgfmathdeclarefunction{f}{1}{\pgfmathparse{0.9 + 0.22*(#1) + 0.6*(1 - exp(-0.5*(#1)))}}
  \pgfmathdeclarefunction{g}{1}{\pgfmathparse{f(\xmax - (#1)) + 0.6}}

  \draw[->] (0,0) -- (\xmax,0);
  \draw (\il,0) -- ++(0,0.06) node[above] {$i_l$};
  \draw (\ir,0) -- ++(0,0.06) node[above] {$i_r$};

  \draw[very thick, smooth] plot[domain=0:\xmax, samples=200] (\x,{g(\x)});
  \pgfmathsetmacro{\vli}{g(\il)}
  \pgfmathsetmacro{\vri}{g(\ir)}
  \fill (\il,\vli) circle (0.8pt) node[above left=2pt] {$v(i_l)$};
  \fill (\ir,\vri) circle (0.8pt) node[above right=2pt] {$v(i_r)$};

  \pgfmathsetmacro{\xturn}{\il + \fracpos*(\ir - \il)}
  \draw[very thick, vmaxcolor] (\il,\vli) -- node[above, pos=0.55] {$v_{\max}$} (\xturn,\vli);
  \draw[very thick, blue]      (\xturn,\vli) -- node[sloped, above, pos=0.55] {$\beta$} (\ir,\vri);
\end{tikzpicture}
\end{minipage}
&
\fbox{\parbox[c][\RowH][c]{0.96\linewidth}{%
  \vspace{\SmallTextHeight}\scriptsize \textit{Case 3 ($\vmax$ on the left): If $v(i_l)=\vmax$ then the tent is capped by a horizontal line of $\vmax$ and has no $\alpha$ component.}\vspace{\SmallTextHeight}%
}}
\\[-2pt]

\begin{minipage}[c][\RowH][c]{\linewidth}\centering
\begin{tikzpicture}[baseline=(current bounding box.center),
                    scale=\diagramScale, x=1cm, y=1cm, >=Stealth, line cap=round,
                    every node/.style={font=\scriptsize}]
  \def\il{1.6}\def\ir{8.2}\def\xmax{9.6}
  \pgfmathsetmacro{\delta}{0.25}
  \pgfmathsetmacro{\tentliftOne}{0.8}
  \pgfmathsetmacro{\tentliftFour}{0.7}
  \pgfmathdeclarefunction{f}{1}{\pgfmathparse{0.9 + 0.22*(#1) + 0.6*(1 - exp(-0.5*(#1)))}}

  \draw[->] (0,0) -- (\xmax,0);
  \draw (\il,0) -- ++(0,0.06) node[above] {$i_l$};
  \draw (\ir,0) -- ++(0,0.06) node[above] {$i_r$};

  \draw[very thick, smooth] plot[domain=0:\xmax, samples=200] (\x,{f(\x)});
  \pgfmathsetmacro{\vli}{f(\il)}
  \pgfmathsetmacro{\vri}{f(\ir)}
  \fill (\il,\vli) circle (0.8pt) node[above left=2pt] {$v(i_l)$};
  \fill (\ir,\vri) circle (0.8pt) node[above right=2pt] {$v(i_r)$};

  \pgfmathsetmacro{\ymOne}{\vri + \tentliftOne}
  \pgfmathsetmacro{\cOne}{\ymOne - \vri}
  \pgfmathsetmacro{\cTwo}{(1+\delta)*(\ymOne - \vli)}
  \pgfmathsetmacro{\midxOne}{(\cTwo*\ir + \cOne*\il)/(\cOne + \cTwo)}
  \pgfmathsetmacro{\mBetaTarget}{(\ymOne - \vri)/(\ir - \midxOne)}

  \pgfmathsetmacro{\yapex}{\vri + \tentliftFour}
  \pgfmathsetmacro{\xBraw}{\ir - (\yapex - \vri)/(\mBetaTarget)}
  \pgfmathsetmacro{\xB}{max(\il+0.9, min(\ir-0.6, \xBraw))}
  \pgfmathsetmacro{\xA}{\il + 0.22*(\xB - \il)}

  \draw[very thick, red]      (\il,\vli) -- node[sloped, above, pos=0.55] {$\alpha$} (\xA,\yapex);
  \draw[very thick, vmaxcolor] (\xA,\yapex) -- node[above, pos=0.5] {$v_{\max}$} (\xB,\yapex);
  \draw[very thick, blue]      (\xB,\yapex) -- node[sloped, above, pos=0.55] {$\beta$} (\ir,\vri);
\end{tikzpicture}
\end{minipage}
&
\fbox{\parbox[c][\RowH][c]{0.96\linewidth}{%
  \vspace{\SmallTextHeight}\scriptsize \textit{Case 4 ($\alpha\leadsto\vmax\leadsto\beta)$: When it is not possible to take a trajectory of $\alpha\leadsto\beta$ between $v(i_l)$ and $v(i_r)$ without exceeding $\vmax$, the tent must contain a flat region at $\vmax$ between $\alpha$ and $\beta$.}\vspace{\SmallTextHeight}%
}}
\\

\end{tabular}
\end{minipage}

\endgroup
\end{table}
\caption{The four possible shapes of positive tent}
\end{figure}
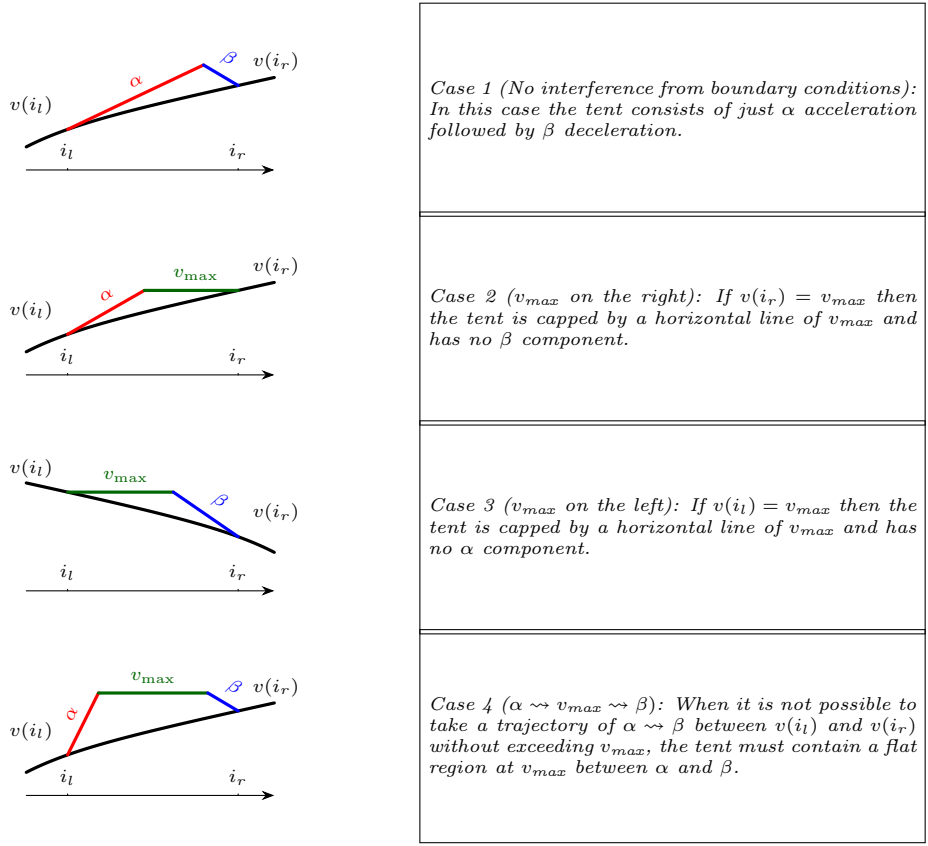
\end{condpics}

\begin{df}
Suppose $0 < i_l < i_r \le q$ and let $\Ic = \{x: i_l \le x \le i_r\}$. We say that a continuous function $v: \Ic \to [0,\vmax]$ is a \emph{negative tent} on $\Ic$ if and only if there exist $i_1$ and $i_2$ such that $i_l \le i_1 \le i_2 \le i_r$ and 
\begin{itemize}
\item on $(i_l,i_1)$, $v$ has slope $-\beta$,
\item on $(i_1,i_2)$, $v$ has slope $0$ and value $0$, and
\item on $(i_2,i_r)$, $v$ has slope $\alpha$.
\end{itemize}
\end{df}

\begin{lm}
For any $0 < i_l < i_r < q$ and values $v_l, v_r \in [0, \vmax]$ such that $-\beta \le \frac{v_r - v_l}{i_r - i_l} \le \alpha$, there exists a unique negative tent $v: (i_l,i_r) \to [0,\vmax]$ such that $v(i_l) = v_l$ and $v(i_r) = v_r$.
\end{lm}
\begin{proof}
$v$ is the pointwise infimum of the family of $(\alpha,\beta)$-Lipschitz functions $f: [i_l, i_r] \to [0,\vmax]$ such that $f(i_l) = v_l$ and $f(i_r) = v_r$. This family is non-empty because the affine linear function satisfying these boundary conditions is $(\alpha,\beta)$-Lipschitz. 
\end{proof}



\colorlet{zerocolor}{orange!85!black} 
\begin{condpics}
\begin{figure}
\begin{table}[H]
\centering
\begingroup
\setlength{\tabcolsep}{2pt}
\renewcommand{\arraystretch}{0.85}

\begin{minipage}{0.82\textwidth}
\centering
\begin{tabular}{@{}M{0.48\linewidth} M{0.48\linewidth}@{}}

\begin{minipage}[c][\RowH][c]{\linewidth}\centering
\begin{tikzpicture}[baseline=(current bounding box.center),
                    scale=\diagramScale, x=1cm, y=1cm, >=Stealth, line cap=round,
                    every node/.style={font=\scriptsize}]
  \def\il{1.6}\def\ir{8.2}\def\xmax{9.6}
  \pgfmathsetmacro{\delta}{0.25}
  \pgfmathsetmacro{\tentdrop}{0.8} 
  \pgfmathdeclarefunction{f}{1}{\pgfmathparse{0.9 + 0.22*(#1) + 0.6*(1 - exp(-0.5*(#1)))}}
  \pgfmathdeclarefunction{g}{1}{\pgfmathparse{f(\xmax - (#1))}} 

  \draw[->] (0,0) -- (\xmax,0);
  \draw (\il,0) -- ++(0,0.06) node[above] {$i_l$};
  \draw (\ir,0) -- ++(0,0.06) node[above] {$i_r$};

  \draw[very thick, smooth] plot[domain=0:\xmax, samples=200] (\x,{g(\x)});
  \pgfmathsetmacro{\vli}{g(\il)}
  \pgfmathsetmacro{\vri}{g(\ir)}
  \fill (\il,\vli) circle (0.8pt) node[above left=2pt] {$v(i_l)$};
  \fill (\ir,\vri) circle (0.8pt) node[above right=2pt] {$v(i_r)$};

  \pgfmathsetmacro{\yapex}{\vri - \tentdrop}
  \pgfmathsetmacro{\cOne}{\vri - \yapex}
  \pgfmathsetmacro{\cTwo}{(1+\delta)*(\vli - \yapex)}
  \pgfmathsetmacro{\midx}{(\cTwo*\ir + \cOne*\il)/(\cOne + \cTwo)}

  \draw[very thick, blue] (\il,\vli) -- 
    node[sloped, above, pos=0.55, xshift=-6pt, yshift=-12pt] {$\beta$} 
    (\midx,\yapex);
  \draw[very thick, red]  (\midx,\yapex) -- 
    node[sloped, above, pos=0.45, yshift=-2pt] {$\alpha$} 
    (\ir,\vri);
\end{tikzpicture}
\end{minipage}
&
\fbox{\parbox[c][\RowH][c]{0.96\linewidth}{%
  \vspace{\SmallTextHeight}\scriptsize \textit{Case 1 (No interference from boundary conditions):
  In this case the tent consists of just $\beta$  deceleration followed by $\alpha$ acceleration.}\vspace{\SmallTextHeight}%
}}
\\[-2pt]

\begin{minipage}[c][\RowH][c]{\linewidth}\centering
\begin{tikzpicture}[baseline=(current bounding box.center),
                    scale=\diagramScale, x=1cm, y=1cm, >=Stealth, line cap=round,
                    every node/.style={font=\scriptsize}]
  \def\il{1.6}\def\ir{8.2}\def\xmax{9.6}
  \pgfmathsetmacro{\fracpos}{0.45}
  \pgfmathdeclarefunction{f}{1}{\pgfmathparse{0.9 + 0.22*(#1) + 0.6*(1 - exp(-0.5*(#1)))}}
  \pgfmathdeclarefunction{g}{1}{\pgfmathparse{f(\xmax - (#1))}} 

  \draw[->] (0,0) -- (\xmax,0);
  \draw (\il,0) -- ++(0,0.06) node[above] {$i_l$};
  \draw (\ir,0) -- ++(0,0.06) node[above] {$i_r$};

  \draw[very thick, smooth] plot[domain=0:\xmax, samples=200] (\x,{g(\x)});
  \pgfmathsetmacro{\vli}{g(\il)}
  \pgfmathsetmacro{\vri}{g(\ir)}
  \fill (\il,\vli) circle (0.8pt) node[above left=2pt] {$v(i_l)$};
  \fill (\ir,\vri) circle (0.8pt) node[above right=2pt] {$v(i_r)$};

  \pgfmathsetmacro{\xmeet}{\il + \fracpos*(\ir - \il)}
  \pgfmathsetmacro{\yzero}{\vri} 

  \draw[very thick, blue]      (\il,\vli) -- 
    node[sloped, above, pos=0.6, xshift=-4pt, yshift=-14pt] {$\beta$} 
    (\xmeet,\yzero);
  \draw[very thick, zerocolor] (\xmeet,\yzero) -- 
    node[above, pos=0.55, xshift=-4pt, yshift=-12pt] {$0$} 
    (\ir,\yzero);
\end{tikzpicture}
\end{minipage}
&
\fbox{\parbox[c][\RowH][c]{0.96\linewidth}{%
  \vspace{\SmallTextHeight}\scriptsize \textit{Case 2 ($0$ on the right): If $v(i_r)=0$ then the tent is capped by a horizontal line of $v=0$ and has no $\alpha$ component.}\vspace{\SmallTextHeight}%
}}
\\[-2pt]

\begin{minipage}[c][\RowH][c]{\linewidth}\centering
\begin{tikzpicture}[baseline=(current bounding box.center),
                    scale=\diagramScale, x=1cm, y=1cm, >=Stealth, line cap=round,
                    every node/.style={font=\scriptsize}]
  \def\il{1.6}\def\ir{8.2}\def\xmax{9.6}
  \pgfmathsetmacro{\fracpos}{0.62}
  \pgfmathdeclarefunction{f}{1}{\pgfmathparse{0.9 + 0.22*(#1) + 0.6*(1 - exp(-0.5*(#1)))}}

  \draw[->] (0,0) -- (\xmax,0);
  \draw (\il,0) -- ++(0,0.06) node[above] {$i_l$};
  \draw (\ir,0) -- ++(0,0.06) node[above] {$i_r$};

  \draw[very thick, smooth] plot[domain=0:\xmax, samples=200] (\x,{f(\x)});
  \pgfmathsetmacro{\vli}{f(\il)}
  \pgfmathsetmacro{\vri}{f(\ir)}
  \fill (\il,\vli) circle (0.8pt) node[above left=2pt] {$v(i_l)$};
  \fill (\ir,\vri) circle (0.8pt) node[above right=2pt] {$v(i_r)$};

  \pgfmathsetmacro{\yzero}{\vli}
  \pgfmathsetmacro{\xturn}{\il + \fracpos*(\ir - \il)}

  \draw[very thick, zerocolor] (\il,\yzero) -- 
    node[above, pos=0.55, yshift=-12pt] {$0$} 
    (\xturn,\yzero);
  \draw[very thick, red]       (\xturn,\yzero) -- 
    node[sloped, above, pos=0.55, xshift=1pt, yshift=-12pt] {$\alpha$} 
    (\ir,\vri);
\end{tikzpicture}
\end{minipage}
&
\fbox{\parbox[c][\RowH][c]{0.96\linewidth}{%
  \vspace{\SmallTextHeight}\scriptsize \textit{Case 3 ($0$ on the left): If $v(i_l)=0$ then the tent is capped by a horizontal line of $v=0$ and has no $\beta$ component.}\vspace{\SmallTextHeight}%
}}
\\[-2pt]

\begin{minipage}[c][\RowH][c]{\linewidth}\centering
\begin{tikzpicture}[baseline=(current bounding box.center),
                    scale=\diagramScale, x=1cm, y=1cm, >=Stealth, line cap=round,
                    every node/.style={font=\scriptsize}]
  \def\il{1.6}\def\ir{8.2}\def\xmax{9.6}
  \pgfmathsetmacro{\tentdropFour}{0.7}
  \pgfmathdeclarefunction{f}{1}{\pgfmathparse{0.9 + 0.22*(#1) + 0.6*(1 - exp(-0.5*(#1)))}}
  \pgfmathdeclarefunction{g}{1}{\pgfmathparse{f(\xmax - (#1))}} 

  \draw[->] (0,0) -- (\xmax,0);
  \draw (\il,0) -- ++(0,0.06) node[above] {$i_l$};
  \draw (\ir,0) -- ++(0,0.06) node[above] {$i_r$};

  \draw[very thick, smooth] plot[domain=0:\xmax, samples=200] (\x,{g(\x)});
  \pgfmathsetmacro{\vli}{g(\il)}
  \pgfmathsetmacro{\vri}{g(\ir)}
  \fill (\il,\vli) circle (0.8pt) node[above left=2pt] {$v(i_l)$};
  \fill (\ir,\vri) circle (0.8pt) node[above right=2pt] {$v(i_r)$};

  \pgfmathsetmacro{\yflat}{\vri - \tentdropFour}
  \pgfmathsetmacro{\xA}{\il + 0.22*(\ir - \il)}
  \pgfmathsetmacro{\xB}{\il + 0.72*(\ir - \il)}

  \draw[very thick, blue]      (\il,\vli) -- node[sloped, above, pos=0.55] {$\beta$} (\xA,\yflat);
  \draw[very thick, zerocolor] (\xA,\yflat) -- node[above, pos=0.5] {$0$} (\xB,\yflat);
  \draw[very thick, red]       (\xB,\yflat) -- 
    node[sloped, above, pos=0.55, xshift = -2, yshift=-2pt] {$\alpha$} 
    (\ir,\vri);
\end{tikzpicture}
\end{minipage}
&
\fbox{\parbox[c][\RowH][c]{0.96\linewidth}{%
  \vspace{\SmallTextHeight}\scriptsize \textit{Case 4 ($\alpha\leadsto0\leadsto\beta)$: When it is not possible to take a trajectory of $\beta\leadsto\alpha$ between $v(i_l)$ and $v(i_r)$ without hitting $v=0$, the tent must contain a flat region at $0$ between $\beta$ and $\alpha$.}\vspace{\SmallTextHeight}%
}}
\\

\end{tabular}
\end{minipage}

\endgroup
\end{table}
\caption{The four possible shapes of negative tents}
\end{figure}
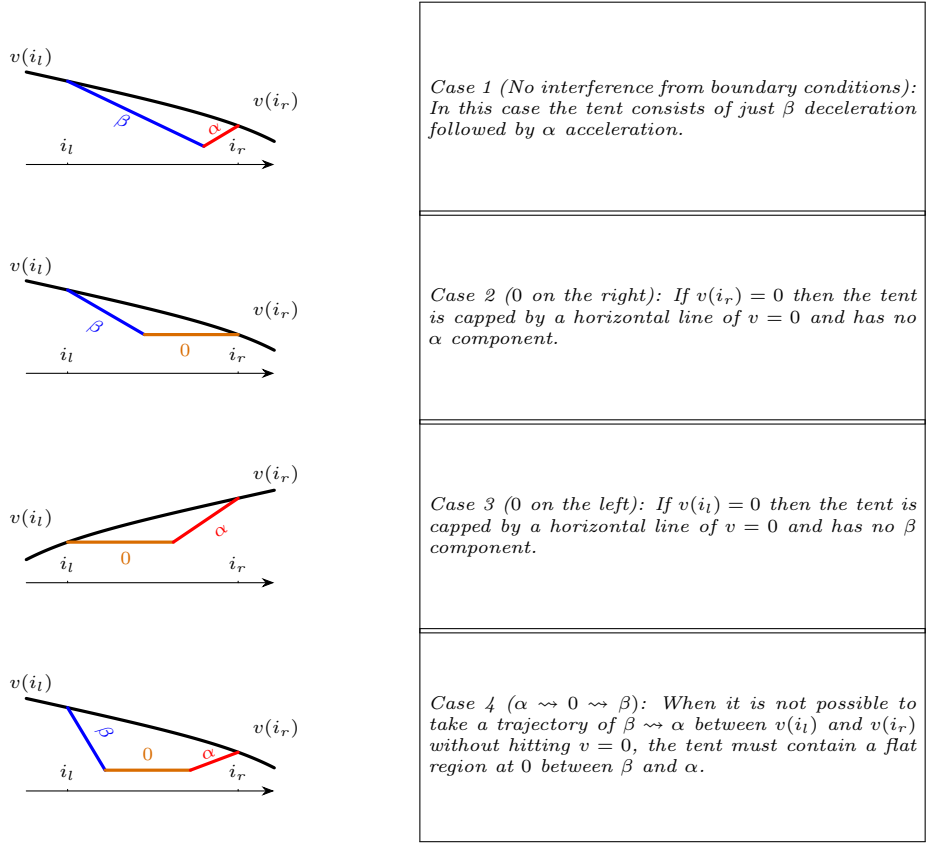
\end{condpics}

\begin{lm}
If $v$ is locally non-increasable on an interval $\Ic \subseteq (0,q)$ then the restriction of $v$ to $\Ic$ is a positive tent.
\label{lm:resofvispos}
\end{lm}
\begin{proof}
Consider an arbitrary closed sub-interval $\Ic$. For every neighbourhood around some $t\in\Ic$, $v$ is concave and hence $v$ is concave on all of $\Ic$. At each point, on each side, \(v\) is either linear with slope \(\alpha\) or \(-\beta\), or constant at level \(v_{\max}\).
Thus the one-sided slopes lie in \(\{\alpha,0,-\beta\}\), with slope \(0\) occurring only on a plateau where \(v\equiv v_{\max}\). Concavity implies one-sided slopes are non-increasing as we move to the right, and at each point the left slope is no less than the right slope. Since allowed slopes are ordered \(\alpha\ge 0\ge -\beta\) and are non-increasing, the slope can drop at most once
from \(\alpha\) to \(0\) and at most once from \(0\) to \(-\beta\).
Hence there exist \(i_1\le i_2\) such that:
\begin{itemize}
  \item on \((i_\ell,i_1)\) the slope is \(\alpha\) (so \(v\) is affine with slope \(\alpha\));
  \item on \((i_1,i_2)\) the slope is \(0\), and \(v\equiv v_{\max}\);
  \item on \((i_2,i_r)\) the slope is \(-\beta\) (so \(v\) is affine with slope \(-\beta\)).
\end{itemize}
Degenerate cases (\(i_1=i_\ell\), \(i_2=i_r\), or \(i_1=i_2\)) are allowed. Therefore \(v|_\Ic\) is a positive tent. 
\end{proof}

\begin{lm}
 If $v$ is not locally non-increasable on some non-empty open set $U$, then there exists a non-empty closed interval $\Ic \subseteq U$ and an $(\alpha,\beta)$-Lipschitz function $v': [0,q) \to [0,\vmax]$ such that $v'(t) = v(t)$ for all $t \notin \Ic$ and the restriction of $v'$ to $\Ic$ is a positive tent and $v'(t) \ge v(t)$ on $\Ic$ and for at least one $t \in \Ic$, $v'(t) > v(t)$. 
 \label{lm:vispos}
\end{lm}

\begin{proof}
Pick $t_0\in U$ where $v$ is not locally non-increasable. Pick a closed interval $\Ic=[i_l,i_r]\subset U$
so that $t_0\in(i_l,i_r)$. Then $v|_{\Ic}$ is not a positive tent, since if it were a positive tent then $v$ would be locally non-increasable at $t_0$. Let $\tau^+$ be the unique positive
tent on $\Ic$ with $\tau^+(i_l)=v(i_l)$ and $\tau^+(i_r)=v(i_r)$, as described in Lemma~\ref{lem:uniquetent}. By the supremum characterisation in Lemma~\ref{lem:uniquetent},
$\tau^+$ is the pointwise supremum of all $(\alpha,\beta)$-Lipschitz functions on $\Ic$ with these endpoint values; since
$v\in V$ is $(\alpha,\beta)$-Lipschitz, we have $v\le \tau^+$ on $\Ic$. If $v\equiv \tau^+$ on $\Ic$ then $v|_{\Ic}$ would be a
positive tent, contradicting the choice of $\Ic$; hence $v<\tau^+$ somewhere in $(i_l,i_r)$. Define
\[
v'(t):=\begin{cases}
\tau^+(t),& t\in[i_l,i_r],\\
v(t),& t\notin[i_l,i_r].
\end{cases}
\]
Then $v'$ is $(\alpha,\beta)$-Lipschitz because both pieces are, and they agree at $i_l,i_r$. Also, $v'$ takes values in $[0,v_{\max}]$,
equals $v$ off $\Ic$, and on $\Ic$ is a positive tent with $v'\ge v$ and $v'>v$ somewhere in $\Ic$.
\end{proof}

\begin{lm}
If $v$ is locally non-decreasable on an interval $\Ic \subseteq (0,q)$ then the restriction of $v$ to $\Ic$ is a negative tent.
\end{lm}
\begin{proof}
The proof of this lemma is almost identical to the proof of lemma~\ref{lm:resofvispos}, simply use convexity to show that the restriction of $v$ to $\Ic$ is a negative tent.
\end{proof}

\begin{lm}
 If $v$ is not locally non-decreasable on some non-empty open set $U$, then there exists a non-empty closed interval $\Ic \subseteq U$ and an $(\alpha,\beta)$-Lipschitz function $v': [0,q) \to [0,\vmax]$ such that $v'(t) = v(t)$ for all $t \notin \Ic$ and the restriction of $v'$ to $\Ic$ is a negative tent and $v'(t) \le v(t)$ on $\Ic$ and for at least one $t \in \Ic$, $v'(t) < v(t)$.
\label{lm:visneg}
\end{lm}
\begin{proof}
The proof of this lemma is also very similar to that of lemma~\ref{lm:vispos}, one can follow the exact same argument with negative tents instead of positive ones.
\end{proof}

\begin{lm}\label{lem:tents_old}
Suppose we have a trajectory $v\in V$  and an Euler--Lagrange curve $\vel$ such that there are two non-empty open subintervals of $(0,q)$, $\mathcal{I} = (i_l, i_r)$ and $\mathcal{J} = (j_l,j_r)$, where $v(t) <\vel(t)$ in $\mathcal{I}$ and $v(t) > \vel(t)$ in $\mathcal{J}$. If $v$ is neither locally non-increasable on $\mathcal{I}$ nor locally non-decreasable on $\mathcal{J}$, then $v$ is not optimal. 
\end{lm}
\begin{proof}
We will construct a perturbation $v'\in V$ of $v$, distinct from $v$, such that $v(t) = v'(t)$ for $t$ outside $\mathcal{I} \cup \mathcal{J}$, $v(t) \le v'(t) \le\vel(t)$ for $ t \in \mathcal{I}$, and $\vel(t) \le v'(t) \le v(t)$ for $t \in \mathcal{J}$. We will then apply Lemma~\ref{lem:movetowardsEL}. 

As we did in the proofs of Lemmas~\ref{lm:vispos} and~\ref{lm:visneg}, let $\tau^+$ be the unique positive tent on $\mathcal{I}$ such that $\tau^+(i_l) =v(i_l)$ and $\tau^+(i_r) = v(i_r)$, and let $\tau^-$ be the unique negative tent on $\mathcal{J}$ such that $\tau^-(j_l) =v(j_l)$ and $\tau^-(i_r) = v(i_r)$.  
Define interpolation constants $\gamma_\mathcal{I},\gamma_\mathcal{J}\in[0,1]$, and define a perturbed trajectory $v'$ by
\[
v'(t)=v(t)+\mathbbm{1}(t\in\mathcal{I})\cdot\gamma_\mathcal{I}[\tau^{+}(t)-v(t)]+\mathbbm{1}(t\in\mathcal{J})\cdot\gamma_\mathcal{J}[\tau^{-}(t)-v(t)].
\]
The interpolation constants $\gamma_\mathcal{I}$ and $\gamma_\mathcal{J}$ must now be chosen to ensure the integral of $v'$ is still equal to $d$. We also choose these small enough such that \(0\le v'(t)\le\vmax\) for all \(t\in(0,\,q)\) and it stays on the same side of \(\vel\) as the unperturbed \(v\). This is possible because of the assumptions that $v$ is neither locally non-increasable on $\Ic$ nor locally non-decreasable on $\Jc$ and that $v < \vel$ on $\Ic$ and $v > \vel$ on $\Jc$. We begin by setting the integral of $v'$ equal to d:
\[
d=\int_0^qv(t)+\mathbbm{1}(t\in\mathcal{I})\cdot\gamma_\mathcal{I}[\tau^{+}(t)-v(t)]+\mathbbm{1}(t\in\mathcal{J})\cdot\gamma_\mathcal{J}[\tau^{-}(t)-v(t)]\di t.
\]
The integral of $v$ is just $d$ since we know that $v\in V$, therefore 
\[
d = d + \gamma_\mathcal{I}\int_\mathcal{I}[\tau^{+}(t)-v(t)]\di t+\gamma_\mathcal{J}\int_\mathcal{J}[\tau^{-}(t)-v(t)]\di t.
\]
Hence we must choose the interpolation constants so that
\[
\gamma_\mathcal{J}=\gamma_\mathcal{I}\cdot\frac{\int_\mathcal{I}[\tau^{+}(t)-v(t)]\di t}{\int_\mathcal{J}[v(t)-\tau^{-}(t)]\di t}. 
\]
Therefore under these conditions $v'\in V$. 

\begin{condpics}
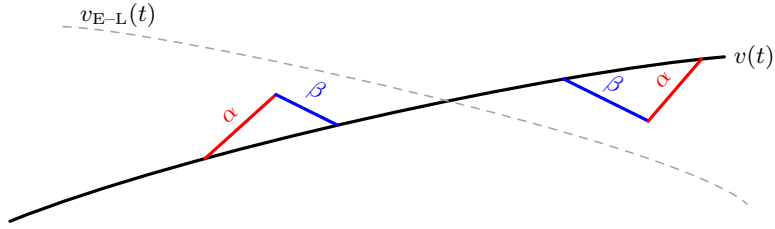
\begin{figure}[H]
\centering
\begin{tikzpicture}[line cap=round, line join=round, font=\small, x=1cm, y=1cm]

  \draw[name path=vt, very thick, black]
    (0.1,0.12) .. controls (2.4,1.05) and (7.2,2.10) .. (9.5,2.30);
  \node[anchor=west] at (9.5,2.30) {$v(t)$};

  \draw[semithick, draw=gray!70, dashed]
    (0.8,2.70) .. controls (2.0,2.68) and (8.8,1.20) .. (9.8,0.35);
  \node[anchor=west] at (0.9,2.85) {$\vel(t)$};

  \coordinate (Lapex) at (3.60,1.80); 
  \path[name path=LalphaRay] (Lapex) -- ($(Lapex)+(-2.2,-2.0)$); 
  \path[name intersections={of=vt and LalphaRay, by=Lstart}];
  \path[name path=LbetaRay]  (Lapex) -- ($(Lapex)+( 2.2,-1.1)$); 
  \path[name intersections={of=vt and LbetaRay,  by=Lend}];

  \draw[very thick, red]  (Lstart) -- (Lapex)
    node[midway, above, sloped] {$\alpha$};
  \draw[very thick, blue] (Lapex)  -- (Lend)
    node[midway, above, sloped] {$\beta$};

  \def\Rxr{9.20} 
  \coordinate (Rapex) at (8.50,1.45);

  \path[name path=RbetaRay] (Rapex) -- ($(Rapex)+(-4,2)$);
  \path[name intersections={of=vt and RbetaRay, by=Rstart}];

  \path[name path=RvR] (\Rxr,-1) -- (\Rxr,3);
  \path[name intersections={of=vt and RvR, by=Rend}];

  \draw[very thick, blue] (Rstart) -- (Rapex)
    node[midway, above, sloped, yshift=-1.5pt] {$\beta$}; 
  \draw[very thick, red]  (Rapex)  -- (Rend)
    node[midway, above, sloped] {$\alpha$};

\end{tikzpicture}
\caption{A candidate trajectory $v(t)$ intersecting an Euler--Lagrange curve. Tents $\tau^{+}$ and $\tau^{-}$ bring the trajectory closer to the Euler--Lagrange curve. An interpolated curve between these tents and $v(t)$ will have a lower action integral.}
\end{figure}
\end{condpics}
Now we can apply Lemma~\ref{interpolatedv} to see that $v$ is not optimal.
\end{proof}

Suppose $v \in V$ is an optimal trajectory. For any $p \in \mathbb{R}$, let
\[
\begin{aligned}
L_{p} &= \{\,t \in (0,q):\; P_B(t,v(t))<p\,\},\\
U_{p} &= \{\,t \in (0,q):\; P_B(t,v(t))>p\,\} .
\end{aligned}
\]

\begin{cor}\label{cor:upperorlower}
Suppose $v \in V$ is an optimal trajectory. Then for all $p \in \mathbb{R}$, either $v$ is locally non-increasable on $L_p$ or $v$ is locally non-decreasable on $U_p$. If p is too  large or small, then $U_p$ or $L_p$ maybe be empty. For such values of $p$ this statement has no content since one of the sets will be empty. 
\end{cor}
\begin{proof}
This is a simple corollary of Lemma~\ref{lem:tents_old}.
\end{proof}
\newcommand{\upperprop}{\hyperref[prop:upper]{Upper property}}
\newcommand{\lowerprop}{\hyperref[prop:lower]{Lower property}}

\begin{pr}
 There exists a pressure value $p_0$ such that $v$ is locally non-increasable on $L_{p_0}$ and locally non-decreasable on $U_{p_0}$. 
\label{prop:upperandlower}
\end{pr}
\begin{proof}
Let
\[
\begin{aligned}
\Ac &= \{\, p : \text{ $v$ is locally non-increasable on $L_p$}  \,\},\\
\Bc &= \{\, p : \text{ $v$ is locally non-decreasable on $U_p$} \,\}.
\end{aligned}
\]
By Corollary~\ref{cor:upperorlower}, $\Ac \cup \Bc =\mathbb{R}$. For any $p\in \Ac$, if $p'<p$ then  $L_{p'} \subseteq L_p$ so $p' \in \Ac$. That is, $\Ac$ is a down-set. Similarly, $\Bc$ is an up-set. We claim that $\Ac$ and $\Bc$ are closed sets. Indeed, suppose for a contradiction that $\Ac$ is not closed; then $\Ac$ must be an interval of the form $(-\infty, p)$ for some $p \in \mathbb{R}$. Then $L_p = \bigcup_{p' < p} L_{p'}$ since for every $t \in L_p$, we may choose $p'$ such that $ P_\Bc(t,v(t)) < p' < p$, then $t \in L_{p'}$, and since $v$ is locally non-decreasable on $L_{p'}$, we see that $v$ is locally non-decreasable at $t$. Hence $v$ is locally non-decreasable on $L_p$, which is to say that $p \in \Ac$, a contradiction. Similarly, $\Bc$ is closed.
By definition, $v$ is locally non-increasable on the empty set. Recalling the pressure bounds~\eqref{eq:boundedp}, if we take $p = -\vmax K + B$ then $L_p = \emptyset$, so $-\vmax K + B \in \Ac$. That is, $\Ac$ is non-empty.  Similarly, $\Bc$ is non-empty, since $\al+B \in \Bc$.

It now follows that $\Ac \cap \Bc \neq \emptyset$. We may take any $p_0 \in \Ac \cap \Bc$, then $v$ is locally non-increasable on $L_{p_0}$ and locally non-decreasable on $U_{p_0}$, as required.
\end{proof}

We can now combine these lemmas to show that the optimal trajectory must be one of the seventeen choices outlined in Theorem~\ref{thm:proofofcombinatorics}. At this point we will define a useful new quantity $v_\beta$, the point at which the deceleration on an Euler--Lagrange curve becomes greater than $\beta$, so that the optimal curve can no longer follow Euler--Lagrange without violating boundary conditions. 

\begin{proof}[Proof of Theorem~\ref{thm:proofofcombinatorics}] Let $v^* \in V$ be an optimal trajectory.
Apply Proposition \ref{prop:upperandlower} to fix a pressure value $p_0$. Recall that Euler--Lagrange curves are the isobars of pressure and therefore have constant pressure value. Proposition \ref{prop:upperandlower} implies that $v^*$ consists of segments of this particular Euler--Lagrange curve, or pieces below this curve that are locally non-increasable, or pieces above that are locally non-decreasable. By lemmas \ref{lm:vispos} and \ref{lm:visneg}, these pieces must be upper or lower tents, respectively. Notice though that any of these three choices might be missing from $v^*$.

Therefore $v^*$ must be a concatenation of pieces from the finite set $\{\alpha,\beta,\text{E--L},\vmax,0\}$.  (Here E--L stands for Euler--Lagrange.) We will perform a case analysis to construct an exhaustive list of possible optimal trajectories. We begin by supposing that the intersection between $v^*$ and the Euler--Lagrange curve of pressure \(p_0\) is non-empty. This could happen via a positive length-segment of Euler--Lagrange curve or the optimal path just crossing it in isolated points. In this latter case we consider such isolated points as length-zero segments. What could possibly be attached to the left of such a segment? One option is that the leftmost point is simply the start time $0$, or there could be some more pieces of a different class of trajectory that come before this leftmost point. If the piece of trajectory we add on the left is always greater than the Euler--Lagrange curve, then we know it must be a negative tent by Lemma~\ref{lm:visneg}. Of the possible negative tents, only one could be added to the leftmost point of a segment of Euler--Lagrange curve and remain greater than it, and that is just a single piece of $\beta$ deceleration. Now if the piece of trajectory we add on the left is less than the Euler--Lagrange curve, it must be a positive tent by Lemma~\ref{lm:vispos}. Of the positive tents, if we wish to remain less than the Euler--Lagrange curve, we can append either
\begin{itemize}
\item $\vmax$, $\alpha \leadsto \vmax$ or just $\alpha$, if the leftmost point of Euler--Lagrange is at velocity \(v>v_\be\), or
\item any of the upper tents if the leftmost point is at velocity \(v\le v_\be\).
\end{itemize}
Now we proceed with the right side of the Euler--Lagrange segment. Firstly, we have the trivial case that there are no adjoining trajectories and the Euler--Lagrange segment just ends at $v=0$. If there is an adjoining segment on the right, and that segment is greater than the Euler--Lagrange curve then it must be a negative tent as before. To stay above the Euler--Lagrange curve, this negative tent can be either
\begin{itemize}
\item $v=0$, if the velocity at this rightmost point is \(v\ge v_\be\) (i.e.\ \(v_\be<0\)), or
\item \(\be \leadsto 0\), if the velocity at this rightmost point is $v<v_\beta$. This is because here the Euler--Lagrange curve is steeper than $\beta$ deceleration.
\end{itemize}
By putting together these restrictions on a trajectory some simple case analysis produces the list in Theorem~\ref{thm:proofofcombinatorics}. 

\end{proof}

Before we move onto specific cases of distributions we will see that the Lagrange multiplier simply allowed us to shift the pressure values by an arbitrary constant under certain conditions.

\begin{lm}
Varying \(B\) shifts the Euler--Lagrange curves left and right if and only if the time for the light to go green, \(T\), has either the Uniform\((0,\,q)\) or an Exponential distribution.
\label{lm:Bcondition}
\end{lm}

\begin{proof}
We write \eqref{eulerlagrangegeneral} as a function of t and B, 
\[
v =\vel^B(t,B)= \vmax-\frac B{f(t)}-\al\frac{F(t)}{f(t)}.
\]
Computing partial derivatives we have
\[
\begin{aligned}
\frac{\partial v}{\partial B} &= -\frac{1}{f(t)}, \\[6pt]
\frac{\partial v}{\partial t} &= \frac{B \dot f(t) - \alpha \bigl[\dot F(t)f(t) - F(t)\dot f(t)\bigr]}{f(t)^2}.
\end{aligned}
\]
By imposing the condition that $\frac{\partial v}{\partial B}\propto \frac{\partial v}{\partial t}$ and rewriting everything in terms of $F(t)$ and its derivatives we arrive at the following ODE:
\[
c\dot{F}=B\ddot{F}-\alpha\dot{F}^2+\alpha F\ddot{F},
\]
where c is a constant.
Setting $y = \dot{F}$ gives $\ddot{F} = y\,\frac{dy}{dF}$, so that $c = (B + \alpha F)\frac{dy}{dF} - \alpha y$.
Solving the linear equation yields
\[
y = C_1(B + \alpha F) - \frac{c}{\alpha}.
\]
Hence, $\dot{F} = D F + A$, with $D = \alpha C_1$ and $A = C_1 B - \frac{c}{\alpha}$.
One option here is \(C_1=0\), in which case \(y=\dot F\) is a constant and we arrive at the Uniform distribution.
Otherwise, the solution is
$F(t) = \text{const.}\cdot e^{D t} - \frac{A}{D}$. Applying $F(0)=0$ and $F(t)\to1$ as $t\to q$ gives
\[
F(t) = 1 - \exp\!\left( \frac{c}{B + \alpha} t \right),
\quad \frac{c}{B + \alpha} < 0,
\]
ensuring $\dot{F} \ge 0, \, \ddot{F} \le 0$.
\end{proof}

Motivated by this, we will now investigate the Uniform and Exponential distributions, which are clearly both special cases of Lemma~\ref{lm:Bcondition}. The Euler--Lagrange curves are given by the following isobars for an Exponential and Uniform Distribution:
\begin{equation}
 \vel(t)=\vmax-\frac B{f(t)}-\al\frac{F(t)}{f(t)}=\left\{
  \begin{aligned}
   &\vmax+\frac\al\la-\Bigl(\frac{\al+B}{\la}\Bigr)e^{\la t},&&\text{for }T\sim\text{Exp}(\la),\\
   &\vmax-\al t-qB,&&\text{for }T\sim\text U(0,\,q).
  \end{aligned}
 \right.\label{eq:velexpression}
\end{equation}
For these examples, varying \(B\) just pushes the Euler--Lagrange curve left or right, thanks to Lemma~\ref{lm:Bcondition}. In other cases this could be more complicated. If the Euler--Lagrange curve starts at \(v_0=v(0)\) then
\[
 v_0=\vmax-\frac B{f(0)},\qquad B=(\vmax-v_0)f(0)
\]
so
\[
 \vel(t)=\vmax-(\vmax-v_0)\frac{f(0)}{f(t)}-\al\frac{F(t)}{f(t)}=\vmax-\frac{(\vmax-v_0)f(0)+\al F(t)}{f(t)},
\]
which is still non-increasing in \(t\), since \(f\) is non-increasing.

\section{The Uniform Case}

Let us now assume that \(T\sim\text{Uniform}(0,\,q)\) for some \(q>0\), hence \(f(t)=\frac1q\) on this interval. Then ODE \eqref{eq:masterode} simply reads \(\dot v=-\al\). That is, the Euler--Lagrange curves represent deceleration at constant rate \(\al\). We assume throughout this section that $\be \geq \al$, so that deceleration at rate $\alpha$ is allowed. In practice this is a realistic assumption: nearly all road vehicles can brake as fast as they can accelerate.

There are many trajectories that finish with a phase of deceleration at rate \(\al\) ending at time \(q\), and we need to pick one that satisfies $x(0) = 0$, $\dot x = v$ where $v \in V$. 

\subsection{Case analysis of trajectories}
Since $\dot v =-\al\ge-\be$ on all Euler Lagrange curves, we know that the boxed trajectories in \ref{thm:proofofcombinatorics} are not valid for the uniform case since Euler Lagrange curves are clearly never steeper than $\beta$. We also have a finite $q$ in the uniform case, so trajectories which do not end stationary are valid, such as $\al$, $\alpha \leadsto \vmax$ and constant $\vmax$. Recalling the notion of pressure introduced in Section \ref{pressureinterpretation}, we can imagine a tank of water filling according to level lines prescribed by shifted Euler Lagrange curves of the form $\dot v =-\al$. We will use Figure \ref{fig:uniformpressure} to inform our case analysis of optimal trajectories. It is helpful to define four critical time points. Let $t_1$ be the time that an instant deceleration from $v_0$ becomes stationary, $\frac{v_0}{\beta}$. We then define $t_2$ to be the time it takes to decelerate from $v_0$ to stationary on an Euler--Lagrange curve, $\frac{v_0}{\alpha}$.  Next, $t_3$ is defined as the time at which instant acceleration from $v_0$ at rate $\alpha$ reaches velocity $\vmax$, $\frac{\vmax-v_0}{\alpha}$. Finally, $t_4$ is equal to $t_3$ added to the time it takes to decelerate from $\vmax$ to $0$ at rate $\alpha$, $t_3 + \frac{\vmax}{\alpha}$. We know that $t_1\leq t_2\leq t_4$ and $t_3\leq t_4$ always, but $q$ can take take any positive value. This leaves us with 15 possible orderings of the inequality once we include $q$, depending on the values of $\alpha,\beta,v_0$ and $\vmax$. We illustrate one of these cases in Figure \ref{fig:uniformpressure}. A complete case-by-case analysis here would be very cumbersome, so we instead treat a general method for finding the optimal solutions in the uniform case. We have depicted Figure \ref{fig:uniformpressure} at an angle so that the Euler--Lagrange lines are horizontal to reinforce the intuition that finding an optimal solution is alike to filling a tank with water, where level lines are lines of Euler--Lagrange. A simple method for finding the solution is as follows. First fix values of $\alpha$, $\beta$, $v_0$ and $q$. We now know which order all of the time steps $t_i$ are in and can draw the relevant picture. All that remains is to `fill' the tank according to the Euler--Lagrange level lines. Fixing $d$ tells us how much to fill the tank and hence the optimal trajectory is the surface level of the water once the area is equal to $d$.

\subsection{Phase Diagrams - Uniform Case}\label{sec:unifphasediagrams}

It is now time to introduce our first phase diagram, we will use these plots throughout the following sections to illustrate the possible different optimal trajectories depending on the values of $v_0$ and $d$. The case of $t_1\leq t_2\leq t_3 \leq t_4 \leq q$ is depicted in Figure~\ref{fig:uniformphasespacet1t2t3t4q}.
This is just one case of many; other orders of $t_1, \dots, t_4$ give rise to different plots. Each one can be understood by thinking about the water filling the relevant shape of tank: the tank for the case $t_1 \leq t_2\leq t_3 \leq t_4 \leq q$ is shown in Figure~\ref{fig:uniformpressure}. The three regions shown in Figure~\ref{fig:uniformphasespacet1t2t3t4q} are as follows. First, a region marked ``No legal trajectories'', where the traffic light is too close for the car to be able to stop before reaching it. Second, the region enclosed by the two curves and the lines $v_0 = \vmax$ and $v_0 = 0$ represents the phase where the optimal trajectory begins with acceleration at rate $\alpha$ to some velocity less than $\vmax$, followed by an E--L curve (constant deceleration at rate $\alpha$) until coming to a halt. To the right of the green curve is the phase where the initial rate $\alpha$ acceleration continues until velocity $\vmax$ is reached, followed later by switching to an E--L curve until coming to a halt.

\begin{condpics}
\begin{figure}[ht]
\centering
\begin{tikzpicture}
  \pgfmathsetmacro{\vmax}{200}                   
  \pgfmathsetmacro{\alphaparam}{6}               
  \pgfmathsetmacro{\betaparam}{20}               

  \begin{axis}[
    width=12cm, height=8cm,
    xlabel={$d$}, ylabel={$v_0$},
    xmin=0, xmax=8000, ymin=0, ymax=210,
    axis lines=middle,
    legend style={
      at={(1.02,0.98)}, anchor=north west,
      font=\small, draw=none,
      row sep=6pt
    },
    legend cell align=left,
  ]

    \addplot[thick,blue,domain=0:\vmax,samples=200]
      ({ x^2/(2*\betaparam) },{ x });
    \addlegendentry{$d=\displaystyle\frac{v_0^2}{2\,\beta}$}

    \addplot[thick,green!60!black,domain=0:\vmax,samples=200]
      ({ \vmax^2*(1/(\alphaparam)) 
         - x^2/(2*\alphaparam) },
       { x });
    \addlegendentry{%
      $d=\displaystyle
        \frac{v_{\max}^2}{\alpha}
        - \frac{v_0^2}{2\,\alpha}$%
    }

    \addplot[dashed,gray] coordinates {(0,\vmax) (8000,\vmax)};
    \addlegendentry{$v_0 = v_{\max}$}

    \node[gray, anchor=north west]
      at (axis cs:7400,{\vmax-5}) {$v_{\max}$};

    \node[
      font=\small\bfseries,
      anchor=south west
    ] at (axis cs:1600,75) {$\alpha \leadsto EL \leadsto 0$};

    \node[
      font=\small,
      anchor=south west,
      rotate=75
    ] at (axis cs:260,100) {No legal trajectories};

    \node[
      font=\small\bfseries,
      anchor=south west
    ] at (axis cs:5500,130) {$\alpha \leadsto v_{\mathrm{max}}\leadsto EL\leadsto0$};

  \end{axis}
\end{tikzpicture}
\caption{Phase-space diagram of $v_0$ vs.\ $d$ for $t_1 \leq t_2\leq t_3\leq t_4 \leq q$, with parameters $\alpha=6$, $\beta=20$, and $v_{\max}=200$. Units do not reflect realistic physical values.} \label{fig:uniformphasespacet1t2t3t4q}
\end{figure}
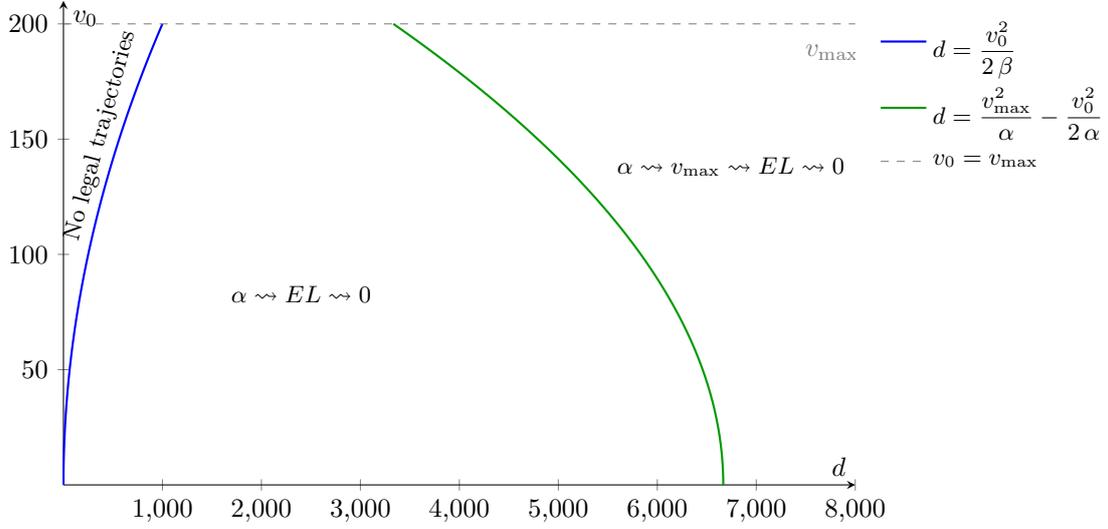
\end{condpics}

\section{The Exponential Case}

In this section we analyse the case where \(Y \sim \text{Exp}(\la)\), so that \(T\sim\text{Exp}(\la)\), \(f(t)=\la e^{-\la t}\) for \(t\ge0\), and we have $q = \infty$. The decision to solve this case is motivated by the memoryless property of the Exponential distribution, which should simplify the problem, since it follows that the optimal acceleration can depend only on the current position and velocity, and not the elapsed time. We have also seen in Lemma~\ref{lm:Bcondition} that the Exponential distributions have the special property that the Euler--Lagrange curves are translates of one another. The ODE in \eqref{eq:masterode} becomes
\[
 \dot v-\la v=-\al-\la\vmax.
\]
A particular solution is \(v(t)=\vmax+\frac\al\la\), while the homogeneous equation \(\dot v-\la v=0\) has general solution \(v(t)=-be^{\la t}\). Combining these gives us a complete general solution of the unconstrained Euler--Lagrange equation: 

\begin{equation}
v(t) = \frac{\al}{\lambda} + \vmax -be^{\lambda t}\label{generalsol}.
\end{equation}

Note that this is consistent with \eqref{eq:velexpression}, where the constant $b$ is written in terms of other constants as $\frac{\al+B}{\la}$. As a first iteration, we are interested in a solution that at time zero starts with \(0<v_0:\,=v(0)\le\vmax\). With this initial condition we have the solution
\begin{equation}
 v(t)=\vmax+\frac\al\la-\Bigl(\vmax-v_0+\frac\al\la\Bigr)e^{\la t}.\label{eq:vsol1}
\end{equation}
At time
\[
 t=\frac1\la\ln\Bigl(\frac{\la\vmax+\al}{\la\vmax-\la v_0+\al}\Bigr)
\]
this velocity becomes zero, after which it turns negative, so this solution cannot be used beyond this time. 

We can also integrate \eqref{eq:vsol1} to get the distance \(x(t)\) travelled by time \(t\):
\[
 x(t)=\int_0^tv(r)\,\di r=\Bigl(\vmax+\frac\al\la\Bigr)t-\Bigl(\frac{\vmax-v_0}\la+\frac\al{\la^2}\Bigr)\cdot\bigl(e^{\la t}-1\bigr).
\]

We can also now solve for $v_\beta$ explicitly. By setting the derivative of the general solution to Euler--Lagrange \eqref{generalsol} equal to $-\beta$, we have that
\[
\dot{v}(t) = -\lambda be^{\lambda t} = -\beta.
\]
Substituting this into \eqref{generalsol} gives us
\[
v_\beta=\vmax + \frac{\alpha-\beta}{\lambda}.
\]

\subsection{When to slam the brakes?}
\label{sec:whentoslamthebrakes}
We saw that, when \(v_\be>0\), it is possible for an optimal trajectory to end with a segment of Euler--Lagrange curve followed by $\beta$ deceleration to come to a halt at the traffic light. But our analysis so far does not tell us the optimal choice of the velocity $v_c$ at which to switch from E--L to $\beta$.  We denote the time when this switch happens by $t_c$, and also define
\begin{equation}
 A=\vmax+\frac{\al}{\lambda}.\label{eq:adef}
\end{equation}
\begin{tm}\label{tm:vc}
Suppose that $v_\beta > 0$. If $v^*$ is an optimal trajectory that ends in $\text{E--L}\leadsto\beta$, then $v^*$ switches from an Euler--Lagrange curve to $\beta$ deceleration upon reaching a velocity $v_c^*$ that is the unique solution to the equation $\mathcal{F}(v_c) = 0$ in the range $v_\beta \le v_c < \vmax$, where
\begin{equation}
\mathcal{F}(v_c) := -\frac{\lambda^2}{\beta}v_c^2 + \frac{\lambda}{\beta}(\beta+\lambda A)v_c + (e^{-\frac{\lambda v_c}{\beta}}-1)(\beta+\lambda A).\label{eq:vc_expr}
\end{equation}
\end{tm}
For certain parameter choices the unique solution of $\mathcal{F}(v_c) = 0$ with $v_c > v_\beta$ may in fact satisfy $v_c \ge \vmax$, in which case there is no optimal trajectory that ends in $\text{E--L}\leadsto\beta$; instead there is no E--L phase phase, and the optimal trajectory must end in $\vmax\leadsto\beta$ or in $\alpha\leadsto\beta$. Which one of these occurs is determined by the distance constraint.

Because the Exponential Distribution is memoryless, the decision about when to switch optimally should only depend on the current velocity and location (i.e.\ volume of `liquid' past the current time); information about the earlier part of the trajectory is irrelevant. In particular, for the derivation of the switching velocity $v_c^*$ it suffices to deal with the case where the  Euler--Lagrange curve was preceded by a \(\vmax\) phase, and that $v_0 = \vmax$. (The other cases must produce the same result.) Therefore for ease of calculation we will progress with the class of trajectories $\vmax \leadsto \text{E--L} \leadsto\beta$.
Recalling the form of \eqref{generalsol} the trajectory we have described can be written as 
\begin{equation}
v(t)=
\begin{cases}
\vmax,            & 0\leq t\leq t_0(t_c),\\
A-(A-v_c)e^{-\lambda(t_c-t)},     & t_0(t_c)\leq t\leq t_c,\\
v_c-\beta (t-t_c),       & t_c\leq t\leq t_c +\frac{v_c}{\beta}.
\end{cases}\label{eq:expoptp}
\end{equation}
The notation $t_0(t_c)$ conveys that $t_0$ is simply a function of $t_c$ thanks to the equation
\begin{equation}
\vmax=A-(A-v_c)e^{-\lambda(t_c-t_0)}.
\label{eq:t_ctot_0}
\end{equation}
We divide the proof into two steps. The first step shows how \eqref{eq:vc_expr} arises, and the second step shows existence and uniqueness of a positive solution. We will call the action \(S_{\exp}\) in this Exponential case (specialising \eqref{eq:swithpressure}).
\begin{proof}[Proof: deriving \eqref{eq:vc_expr}.]
We have that
\begin{align*}
S_{\exp}(t_c,B)=
&\int_{0}^{t_0}\int_{0}^{\vmax} P_B(t,C)\,\di C\,\di t +
\int_{t_0}^{t_c}\int_{0}^{A-(A-v_c)e^{-\lambda(t_c-t)}} P_B(t,C)\,\di C\,\di t\\ &+
\int_{t_c}^{\,t_c+v_c/\beta}\int_{0}^{\,v_c-\beta(t-t_c)} P_B(t,C)\,\di C\,\di t.
\end{align*}
 The constant \(B\) is arbitrary, but we need to optimise our \(\vmax\leadsto\text{E--L}\leadsto\be\) trajectory in the set \(V\). The constant volume \(d\) restriction means that, unless \(v'(t_c)=-\be\), changing \(t_c\) will also cause a translation of the Euler-Lagrange curve, modifying the value of \(t_0\). Naturally, \(v_c\) is also sensitive to this change. Below we implicitly differentiate with respect to $t_c$, with the understanding that \(t_0\) and \(v_c\) all change to keep the trajectory in \(V\), meeting the constraint that the total distance equals \(d\). We will use the notation $'$ to denote the derivative with respect to $t_c$. To find an optimal $t_c$ we set $S'_{\exp} = 0$. (Later we will also check the signs of the derivatives with respect to $t_c$ at the endpoints of the allowed range for $t_c$, namely where $v_c = v_\beta$ or $v_c = \vmax$.)
\begin{align}
0 \;=\; S_{\exp}'
&= \notag\\[2pt]
&\quad t_0'\int_{0}^{\vmax} P_B(t_0,C)\,\di C
\tag{T1}\label{term:T1}\\
&\quad + \int_{0}^{A-(A-v_c)e^{-\lambda(t_c-t_c)}} P_B(t_c,C)\,\di C
\tag{T2}\label{term:T2}\\
&\quad - t_0'\int_{0}^{A-(A-v_c)e^{-\lambda(t_c-t_0)}}P_B(t_0,C)\,\di C
\tag{T3}\label{term:T3}\\
&\quad + \int_{t_0}^{t_c}[v_c'+\lambda(A-v_c)]e^{-\lambda(t_c-t)}\cdot
P_B(t,A-(A-v_c)e^{-\lambda(t_c-t)})\,\di t
\tag{T4}\label{term:T4}\\
&\quad + \bigl(1+\tfrac{v_c'}{\beta}\bigr)
\int_0^{v_c-\beta(t_c+\frac{v_c}{\beta}-t_c)}P_B(t_c+\tfrac{v_c}{\beta},C)\,\di C
\tag{T5}\label{term:T5}\\
&\quad - \int_0^{v_c-\beta(t_c-t_c)}P_B(t_c,C)\,\di C
\tag{T6}\label{term:T6}\\
&\quad + (v_c'+\beta)\int_{t_c}^{\,t_c+v_c/\beta}
P_B(t,v_c-\beta(t-t_c))\,\di t
\tag{T7}\label{term:T7}
\end{align}
 There are many cancellations here. \eqref{term:T1} and \eqref{term:T3} cancel since the upper limit of the integrals are both $\vmax$ by \eqref{eq:t_ctot_0}. \eqref{term:T2} and \eqref{term:T6} cancel, since both of their upper limits simplify to $v_c$. \eqref{term:T5} is equal to $0$ since both the limits of the integral are $0$. This leaves us with \eqref{term:T4} and \eqref{term:T7}.

We now turn the generic formula \eqref{eq:pdef} into the particular pressure for the Exponential case:
\[
 P_B(t,\,C)=-(\vmax - C)\la e^{-\la t}+\al\bigl(1-e^{-\la t}\bigr)+B.
\]
 If we expand \eqref{term:T4}, using abbreviation \eqref{eq:adef}, it can be written as
\[
 \begin{aligned}
  &\quad\int_0^{t_c}[v_c'+\lambda(A-v_c)]e^{-\lambda(t_c-t)}\cdot\bigl[-\bigl(\vmax-A+(A-v_c)e^{-\lambda(t_c-t)}\bigr)\la e^{-\la t}+\al\bigl(1-e^{-\la t}\bigr)+B\bigr]\di t.\\
  &=\int_0^{t_c}[v_c'+\lambda(A-v_c)]e^{-\lambda(t_c-t)}\cdot [\al+B-\la(A-v_c)e^{-\lambda t_c}]\di t.
 \end{aligned}
\]
 We still have not fixed $B$; it is just a constant that we can use to shift isobar lines up and down. As a result of the volume boundary condition $\int_0^qv(t) \di t = d$, $S_{\exp}'(t_c,B)-S_{\exp}'(t_c,0)=(B\cdot d-0 \cdot d)'=0$, so $S_{\exp}'(t_c,B)$ does not depend on $B$. If we fix $B=\la(A-v_c)e^{-\lambda t_c}-\al$ it is clear that \eqref{term:T4} is equal to $0$. Our optimality condition now is only that \eqref{term:T7} must be equal to $0$ for this value of $B$.


 We will examine the sign of \(v_c'+\be\) a bit later, but now consider \eqref{term:T7} without this factor. Expanding  the integrand in \eqref{term:T7} we obtain 
\begin{equation}\label{eqn:T7settozero}
 \begin{aligned}
  \int_{t_c}^{t_c+\frac{v_c}{\beta}} P_B(t,v_c - \beta(t-t_c))\di t&=\int_{t_c}^{t_c+\frac{v_c}{\beta}}-\bigl(\vmax - v_c+\beta(t-t_c)\bigr)\la e^{-\la t}+\al\bigl(1-e^{-\la t}\bigr)+B\di t\\
  &=\int_{t_c}^{t_c+\frac{v_c}{\beta}}\la(v_c-\beta(t-t_c)-A)e^{-\lambda t}+\al +B \di t.
 \end{aligned}
\end{equation}
The physical intuition behind equating this to $0$ is that otherwise the slope $-\beta$ wall of our imaginary tank of liquid would feel an overall force, the integral of the pressure along the wall, pushing it to translate as a rigid plate that remains at slope $-\beta$. The pressure along the E--L part of the liquid boundary is zero, by our choice of $B$, so this part of the boundary feels no force.  Although the pressure integrals along the $\alpha$ and $\vmax$ parts of the tank boundary are positive, these boundaries are already at the extremal positions allowed by the constraints. 

Performing the integral~\eqref{eqn:T7settozero}, we obtain
\begin{multline*}
 (v_c+\beta t_c-A)\bigl(e^{-\lambda t_c}-e^{-\lambda(t_c+\frac{v_c}{\beta})}\bigr)+\beta\bigl((t_c+\frac{v_c}{\beta})e^{-\lambda(t_c+\frac{v_c}{\beta})}-t_ce^{-\lambda t_c}\bigr)\\
 +\frac{\beta}{\lambda}\bigl(e^{-\lambda(t_c+\frac{v_c}{\beta})}-e^{-\lambda t_c}\bigr) + (\al+B)\frac{v_c}{\beta}.
\end{multline*}
 The Euler--Lagrange curve at $v_c$ gives us $v_c=A-\frac{\al+B}\la e^{\lambda t_c}$ which rearranges to $\al+B=\la(A-v_c)e^{-\lambda t_c}$. We can substitute this into the expression above to eliminate exponentials involving $t_c$, obtaining
\[
 (v_c+\beta t_c-A)\bigl(1-e^{-\lambda\frac{v_c}{\beta}}\bigr)+\beta\bigl((t_c+\frac{v_c}{\beta})e^{-\lambda\frac{v_c}{\beta}}-t_c\bigr)+\frac{\beta}{\lambda}\bigl(e^{-\lambda\frac{v_c}{\beta}}-1\bigr) + \la(A-v_c)\frac{v_c}{\beta}.
\]

Multiplying through by $\lambda$, this simplifies to
\[
\lambda(v_c-A)-\beta+\frac{\lambda^2}{\beta}(A-v_c)v_c+e^{-\lambda \frac{v_c}{\beta}}(\beta+\lambda A).
\]
Rearranging these terms, we arrive at the expression for $\Fc(v_c)$.
\end{proof}

We record the final simplified formula for $S_{\exp}'$ obtained in the proof above when the appropriate choice of $B$ is made so that the Euler-Lagrange curve is a zero-pressure isobar:
\[ S_{\exp}' = \left(\frac{v_c' + \beta}{\beta}\right)\left(\frac{\beta}{\lambda}\mathcal{F}(v_c)\right).\]
This has a physical interpretation: in the tank-and-water picture $S_{\exp}$ is the potential energy, the term $\left(\frac{v_c'+\beta}{\beta}\right)$ is the rate at which the plate of slope $-\beta$ moves horizontally in response to varying $t_c$, and the term $\beta\int_{t_c}^{t_c + v_c/\beta} P_B(t,v_c - \beta(t-t_c))\di t = \frac{\beta}{\lambda}\mathcal{F}(v_c)$ represents the negative of the horizontal component of the force exerted by the water on the plate (which itself can be negative because of our choice of $B$).

\begin{proof}[Proof of existence and uniqueness of a solution $v_c^*$ for \eqref{eq:vc_expr}]
\begin{figure}[htbp]
  \centering
  \includegraphics[width=0.8\textwidth]{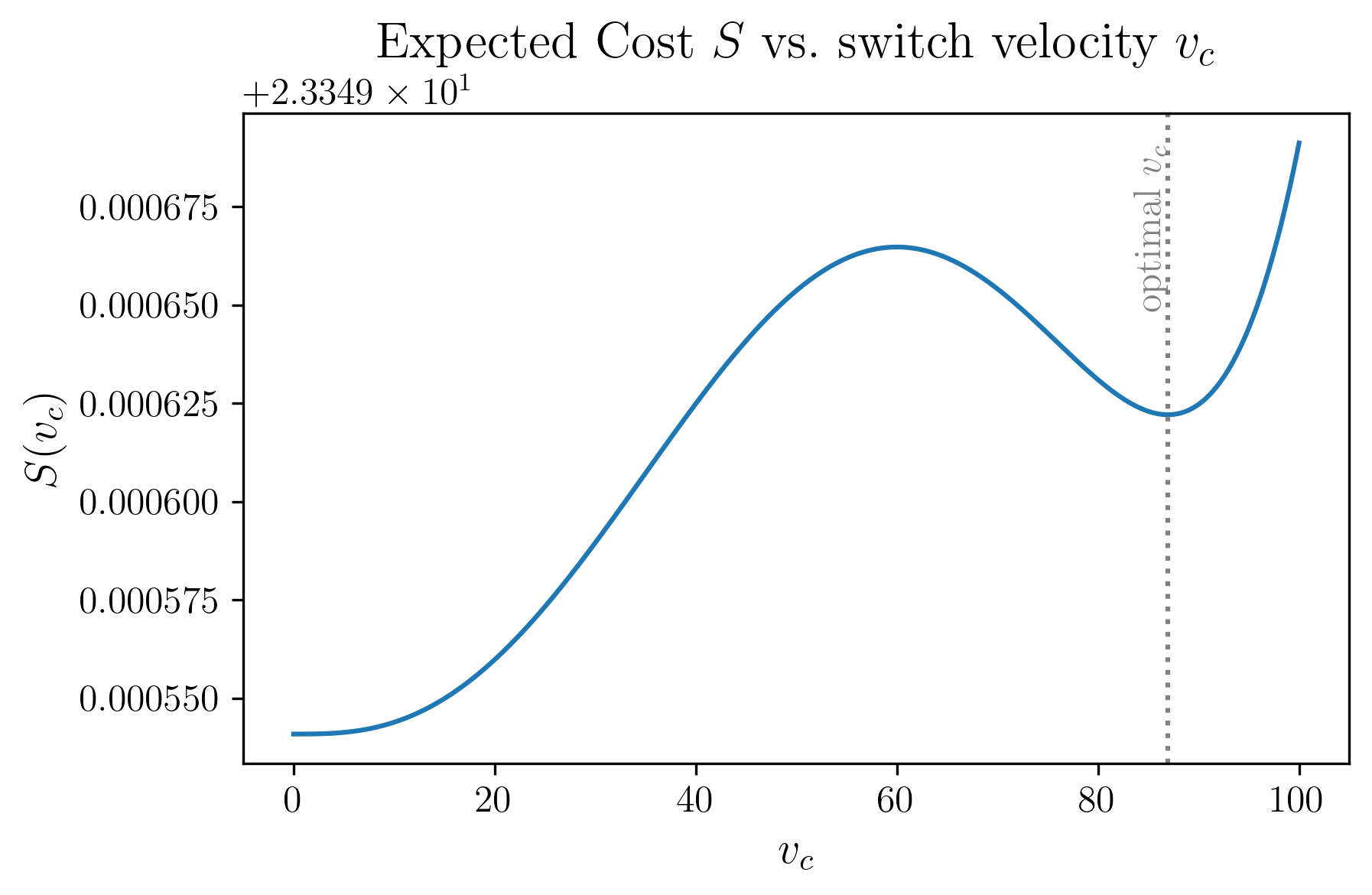}
  \caption{Expected cost \(S(v_c)\) as a function of switch velocity \(v_c\), with parameters
$\lambda=0.1$, $\alpha=6$, $\beta=20$, $d=4000$, $L=4000$ and $v_{\max}=200$. The optimal $v_c\approx86.94$ is given by the local minimum. As we will see slightly later, a restriction \(v_c\ge v_\be\) does not allow us to exploit the global minimum on the left of the curve.}
\end{figure}

To show that such a root exists, we must perform some basic analysis of the function. First we take derivatives.

\[
\Fc'(v_c)=-\frac{2\lambda^2v_c}{\beta} + \frac{\lambda}{\beta}(\beta +\lambda A)(1-e^{-\frac{\lambda v_c}{\beta}})
\]

and 

\[
\Fc''(v_c) = -\frac{2\lambda^2}{\beta} + \frac{\lambda^2}{\beta^2}(\beta+\lambda A)e^{-\frac{\lambda v_c}{\beta}}.
\]

\begin{pr}
For $v_\beta>0$, $v_c^*\geq v_\beta$, where $v_c^*$ is the optimal change velocity from an Euler--Lagrange curve to $\beta$ deceleration and $v_\beta$ is the point at which an Euler lagrange curve decelerates at $\beta$.
\end{pr}
\begin{proof}
Taking the value of the second derivative of $\Fc(v_c)$ at 0 we find
\[
\Fc''(0)=\frac{\lambda^2}{\beta^2}(-\beta+\lambda A)=\frac{\lambda^3}{\beta^2}\left(\vmax+\frac{\alpha-\beta}{\lambda}\right)=\frac{\lambda^3}{\beta^2}v_\beta.
\]

We also see that $\Fc(0)=\Fc'(0)=0$ trivially. As $v\rightarrow \infty$, $\Fc(v_c)\rightarrow -\infty$ and as $v\rightarrow -\infty$, ${\Fc}(v_c)\rightarrow \infty$. For $v_\beta>0$, $\Fc''(0)>0$ and $\Fc''(v_c)\searrow-\frac{2\lambda^2}{\beta}<0$,  so we indeed must have a positive non-trivial root which we called $v_c^*$. It also follows that $\Fc(v_c)>0$ on the interval $(0,v_c^*)$ and $\Fc(v_c)<0$ on the interval $(v_c^*,\infty)$. Similarly, if $v_\beta \leq0$, $\Fc''(0)\leq0$ so we must have one trivial non-negative root.

\begin{condpics}
\begin{figure}[htbp]
  \centering
  \includegraphics[width=0.8\textwidth]{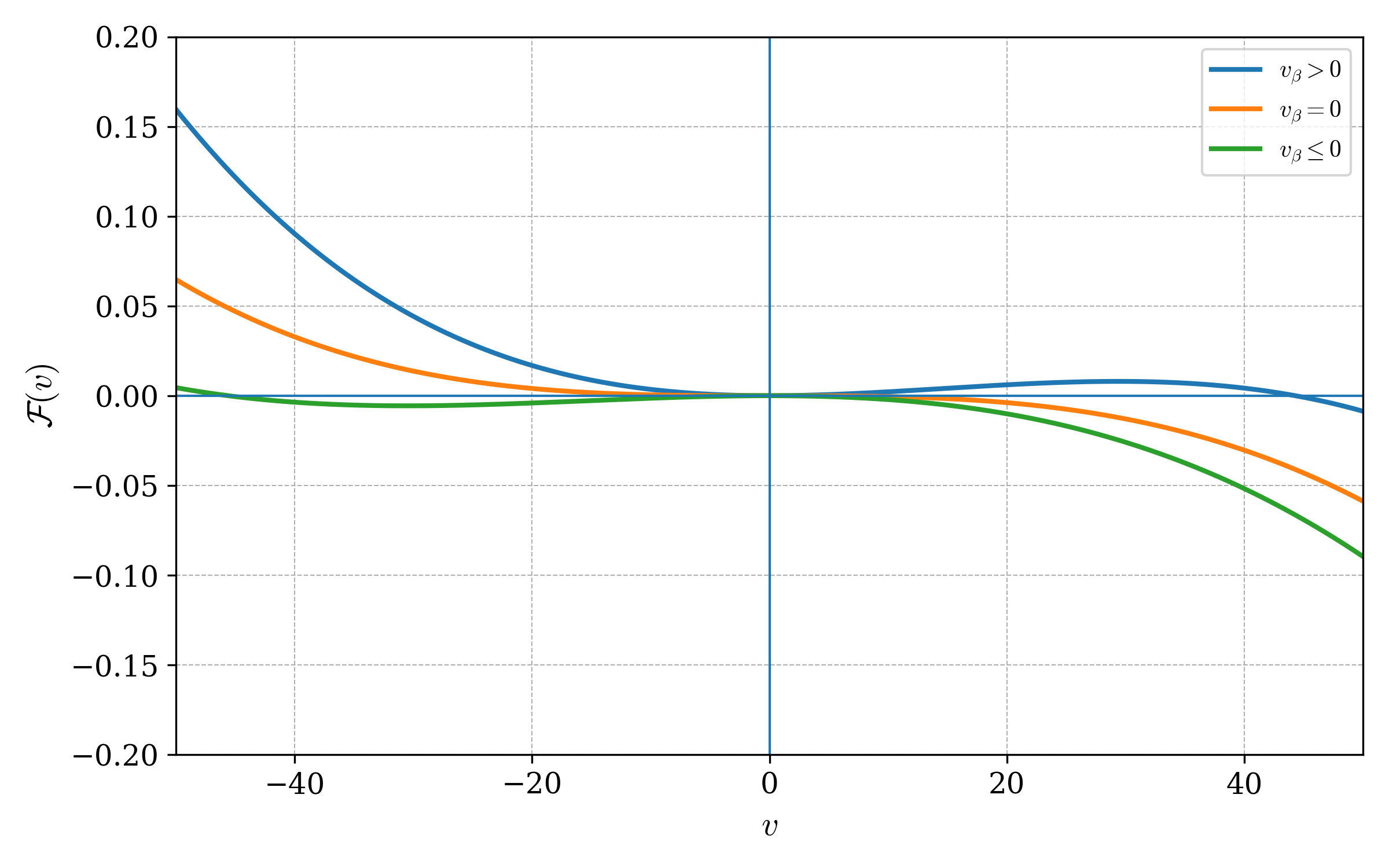}
  \caption{$\Fc(v)$ against $v$ for different values of $v_\beta$}
\end{figure}
\end{condpics}

Now we show that $v_c^*\geq v_\beta$ always.
We define the function
\[
\Gc(U,A):\,=-\beta U^2 + (\beta+\lambda A)(U+e^{-U}-1).
\]
and notice that at $U=\frac{\lambda v_c}{\beta}$ this takes the value of $\Fc(v_c)$.
Next we define 
\[
\Hc(U) :\,= \Gc(U,\frac{\beta}{\lambda}(U+1)) = -\beta U^2 + (\beta + \beta(U+1))(U+e^{-U}-1)= \beta((U-2)+(U+2)e^{-U}).
\]
We can now deduce that
\begin{itemize}
\item $\frac{1}{\beta}\Hc(U)\Big|_{U=0}=0$
\item $\frac{1}{\beta}\Hc'(U)\Big|_{U=0}=1+e^{-U}-U e^{-U}-2e^{-U}\Big|_{U=0}=1-U e^{-U}-e^{-U}\Big|_{U=0}=0$
\item $\frac{1}{\beta}\Hc''(U)=-e^{-U}+Ue^{-U}+e^{-U}=Ue^{-U}>0$ if $U>0$.
\end{itemize}
These three statements imply that $\Hc(U)$ is positive for any $U>0$. In particular, picking $U= \frac{\lambda v_\beta}{\beta}$ gives 
\[
0<\Hc\Bigl(\frac{\lambda v_\beta}{\beta}\Bigr)=\Gc\Bigl(\frac{\lambda v_\beta}{\beta},v_\beta+\frac{\beta}{\lambda}\Bigr)= \Gc\Bigl(\frac{\lambda v_\beta}{\beta},A\Bigr)=\Fc(v_\beta).
\]
Therefore, for $v_\beta>0$, $\Fc(v_\beta)>0$. In view of the above properties of $\Fc(v)$, we must have $v_c^*\geq v_\beta$.
\end{proof}

To connect this result to \(S_{\exp}'\), we now carefully examine the factor \(v_c'+\be\) in \eqref{term:T7}.
\begin{lm}
 \[
  v'_c=-\frac{\la\be\vmax(A-v_c)}{\be(\vmax-v_c)+\la(A-v_c)v_c}.
 \]
\end{lm}
\begin{proof}
 The value of \(v_c\) is determined by the constraint \(\int_0^\infty v(t)\di t=d\). Substituting in from \eqref{eq:expoptp},
 \[
  \int_0^{t_0}\vmax\di t+\int_{t_0}^{t_c}A-(A-v_c)e^{-\la(t_c-t)}\di t+\int_{t_c}^{t_c+v_c/\be}v_c-\be(t-t_c)\di t=d.
 \]
 Differentiating with respect to \(t_c\),
 \[
  \vmax t_0'+v_c-\bigl(A-(A-v_c)e^{-\la(t_c-t_0)}\bigr)t_0'+\int_{t_0}^{t_c}\bigl(v_c'+\la(A-v_c)\bigr)e^{-\la(t_c-t)}\di t-v_c+\int_{t_c}^{t_c+v_c/\be}v_c'+\be\di t=0.
 \]
 The \(t_0'\) terms cancel via \eqref{eq:t_ctot_0}:
 \[
  v_c'\Bigl(\int_{t_0}^{t_c}e^{-\la(t_c-t)}\di t+\frac{v_c}\be\Bigr)+\int_{t_0}^{t_c}\la(A-v_c)e^{-\la(t_c-t)}\di t+v_c=0.
 \]
 Notice that, again via \eqref{eq:t_ctot_0},
 \[
  \la\int_{t_0}^{t_c}e^{-\la(t_c-t)}\di t=1-e^{-\la(t_c-t_0)}=\frac{\vmax-v_c}{A-v_c}.
 \]
 Hence the above display reads
 \[
  v_c'\Bigl(\frac1\la\frac{\vmax-v_c}{A-v_c}+\frac{v_c}\be\Bigr)+\vmax=0.
 \]
\end{proof}
\begin{cor}\label{cr:vcsign}
 When \(0\le v_c\le\vmax\), we have \(v_c'<0\).
\end{cor}
\begin{cor}
 When \(v_\be>0\), we have \(v_c'+\be>0\) in the range \(v_\be<v_c<\vmax\); \(v_c'+\be=0\) at \(v_\be\) and \(\vmax\) and negative outside this interval.
\end{cor}
\begin{proof}
 \(v_c'+\be>0\) is equivalent to
 \[
   \frac{\la\vmax(A-v_c)}{\be(\vmax-v_c)+\la(A-v_c)v_c}<1,
   \]
   i.e.
   \[
   \la v_c^2+(\be-\la A-\la\vmax)v_c+\la\vmax A-\be\vmax<0.
 \]
 The two roots of this convex quadratic are
 \begin{multline*}
  \frac{\la A+\la\vmax-\be\pm\sqrt{(\be-\la A-\la\vmax)^2-4\la(\la\vmax A-\be\vmax)}}{2\la}\\
   =\frac{\la A+\la\vmax-\be\pm(\la\vmax+\be-\la A)}{2\la}=\left\{
    \begin{aligned}
     &\vmax\\
     &A-\frac\be\la=v_\be.
    \end{aligned}
   \right.
 \end{multline*}
 The inequality therefore holds on the interval \((v_\be,\,\vmax)\).
\end{proof}
We can now complete the proof of Theorem \ref{tm:vc}. \(S_{\exp}'\) was shown to be the product of positive constants, \(\Fc(v_c)\) and the factor \(v_c'+\be\). We also saw that, when \(v_\be>0\), we have the following signs:
\begin{center}
 \begin{tabular}{|c|c|c|c|}
  \hline
  &\(0<v_c<v_\be\)&\(v_\be<v_c<v_c^*\)&\(v_c^*<v_c<\vmax\)\\
  \hline
  \(\Fc(v_c)\):&\(+\)&\(+\)&\(-\)\\
  \hline
  \(v_c'+\be\):&\(-\)&\(+\)&\(+\)\\
  \hline
  \((v_c'+\be)\cdot\Fc(v_c)\):&\(-\)&\(+\)&\(-\)\\
  \hline
 \end{tabular}
\end{center}
By Corollary \ref{cr:vcsign} the sign of $\frac{\partial S_{\exp}}{\partial v_c}$ is the opposite of the sign of \(S_{\exp}'\). Thus we see that \(v_\be\) is a local maximum and, if \(v_c^*<\vmax\), then this is a unique minimum in the permissible range \((v_\be,\,\vmax)\).
\end{proof}

In the next section we will show phase space diagrams that divide the $(v_0,d)$-plane into regions corresponding to different structures of optimal trajectory. To do this we will need to create a refinement of the list of optimal trajectories in Theorem~\ref{thm:proofofcombinatorics}. We now know that the choice of these trajectories depends on the values of $v_c^*$ and $v_\beta$. For $v_c^*\geq \vmax$ the optimal solution curve will never include an E--L segment and for $v_\beta\leq0$ it will never change from E--L to \(\be\) deceleration. 
\begin{figure}[!ht]
\begin{center}
 \begin{minipage}[t]{0.3\linewidth}
  {\bf Starting with \(\al\)}
  \begin{itemize}
   \item $\alpha \leadsto \beta$
   \item $\alpha \leadsto EL \leadsto \beta$
   \item $\alpha \leadsto \vmax \leadsto EL \leadsto \beta$
   \item $\alpha \leadsto EL$ {\scriptsize($v_{\beta}\le 0$)}
   \item $\alpha \leadsto \vmax \leadsto EL$ {\scriptsize($v_{\beta}\le 0$)}
   \item $\alpha \leadsto \vmax \leadsto \beta$ {\scriptsize($v_c^*\ge \vmax$)}
  \end{itemize}
 \end{minipage}
 \hskip1cm
 \begin{minipage}[t]{0.3\linewidth}
  {\bf Starting with \(\be\)}
  \begin{itemize}
   \item $\beta$
   \item $\be \leadsto EL \leadsto \beta$
   \item $\beta \leadsto EL$ {\scriptsize($v_{\beta}\le 0$)}
  \end{itemize}
 \end{minipage}
 
 \bigskip\noindent
 \begin{minipage}[t]{0.3\linewidth}
  {\bf Starting with \(\vmax\)}
  \begin{itemize}
   \item $\vmax \leadsto EL \leadsto \beta$
   \item $\vmax \leadsto EL$ {\scriptsize($v_{\beta}\le 0$)}
   \item $\vmax \leadsto \beta$ {\scriptsize($v_c^*\ge \vmax$)}
  \end{itemize}
 \end{minipage}
 \hskip1cm
 \begin{minipage}[t]{0.3\linewidth}
  {\bf Starting with E--L}
  \begin{itemize}
   \item $EL \leadsto \beta$
   \item $EL$ {\scriptsize($v_{\beta}\le 0$)}
  \end{itemize}
 \end{minipage}
\end{center}
\caption{The possible combinations making up an optimal trajectory (Exponential case)}
\end{figure}

\subsection{Phase Diagrams - Exponential Case}
\label{sec:expphasediagrams}
We can now construct our first phase space diagram under the conditions that $v_c^*<\vmax$ and $v_\beta>0$. Firstly, we must find expressions for the boundaries between each phase. The diagram must start with the boundary condition $d\geq \frac{v_0^2}{2\beta}$, since otherwise it is not physically possible for the car to decelerate to a standstill without crashing into the traffic light. If the car starts at $v_0$ below $v_c^*$, we know that it is never optimal to travel along an Euler--Lagrange curve, since the solution curve would have changed to $\beta$ deceleration by this point. The only possible combination of trajectories with no E--L segment and $v_0<v_c^*<\vmax$ is $\alpha \leadsto \beta$ since only $\beta$ deceleration is the boundary we have just discussed. This $\alpha \leadsto \beta$ phase will continue until we cross $v_c^*$ somewhere in the motion, allowing us to solve for another clear boundary. Combining the distance covered accelerating from $v_0$ to $v_c^*$ at $\alpha$, $\frac{(v_c^*)^2-v_0^2}{2\alpha}$, with the distance covered decelerating from $v_c^*$ to a standstill, $\frac{(v_c^*)^2}{2\beta}$, we have $d=(v_c^*)^2(\frac{1}{2\beta}+\frac{1}{2\alpha})-\frac{v_0^2}{2\alpha}$ as our next boundary. Above this boundary, if $v$ is below an E--L curve, the vehicle will accelerate at $\alpha$ to some value $v_a>v_c^*$, and then travel along an E--L curve since this is always the most optimal above $v_c^*$, until it reaches $v_c^*$, after which it must return to $\beta$ deceleration until stationary. For parameters $(d,v_0)$ on this boundary, the optimal solution is to accelerate at rate $\alpha$ until the velocity reaches $v_c^*$ and then immediately slam the brakes to decelerate at rate $\beta$ until coming to a halt. The next boundary we seek is when $v_a$ can increase no further since it is equal to $\vmax$. The $\alpha$ acceleration from $v_0$ to $v_a$ covers $\frac{v_a^2-v_0^2}{2\alpha}$ metres. Next, we must fit an E--L curve to start at $v_a$. The standard E--L curve we derived earlier is given by $v(t)=A-be^{\lambda t}$, where $A=\vmax+\frac{\alpha}{\lambda}$ and we can vary $b$. Using the initial conditions $v(0)=v_a$ we can deduce that $A-b=v_a$ and thus $v(t)=A-(A-v_a)e^{\lambda t}$. Now we compute the distance travelled from $v_a$ to $v_c^*$, which occurs at $t_c^*$.
\[
\int_0^{t_{c}^*} A-(A-v_a)e^{\lambda t}\di t = At_c^* - \frac{A-v_a}{\lambda}(e^{\lambda t_c^*}-1).
\]
Since we have that $v(t_c^*)=v_c^*=A-(A-v_a)e^{\lambda t_c^*}$ we can solve for $t_c^*=\frac{1}{\lambda}\ln\frac{A-v_c^*}{A-v_a}$. Substituting this into the integral result gives
\begin{equation}
\frac{A}{\lambda}\ln\frac{A-v_c^*}{A-v_a} - \frac{v_a-v_c^*}{\lambda}.\label{eq:ELintegral}
\end{equation}
Finally, we can add the distance covered decelerating from $v_c^*$ to stationary at $\beta$, $\frac{(v_c^*)^2}{2\beta}$, and set all $v_a=\vmax$ to obtain a third boundary expression.
\[
d = \frac{\vmax^2-v_0^2}{2\alpha} + \frac{A}{\lambda}\ln\frac{A-v_c^*}{A-\vmax} - \frac{\vmax-v_c^*}{\lambda} + \frac{(v_c^*)^2}{2\beta}
\]
One final boundary expression can be obtained by considering the trajectory for which the vehicle travels down an E--L curve immediately from $v_0$, and then after $v_c^*$ decelerates at a rate of $\beta$ until stationary. Between this curve and the $d = \frac{v_0^2}{2\beta}$ boundary we encounter $\beta \leadsto EL \leadsto \beta$ trajectories and below this curve we encounter the $\alpha \leadsto EL \leadsto \beta$ which we have already explored. By combining the integral of an E--L curve from $v_c^*$ to $v_0$ and $\beta$ deceleration $\frac{(v_c^*)^2}{2\beta}$ we obtain
\[
d =  \frac{A}{\lambda}\ln\frac{A-v_c^*}{A-v_0} - \frac{v_0-v_c^*}{\lambda} + \frac{(v_c^*)^2}{2\beta}.
\]
\begin{condpics}
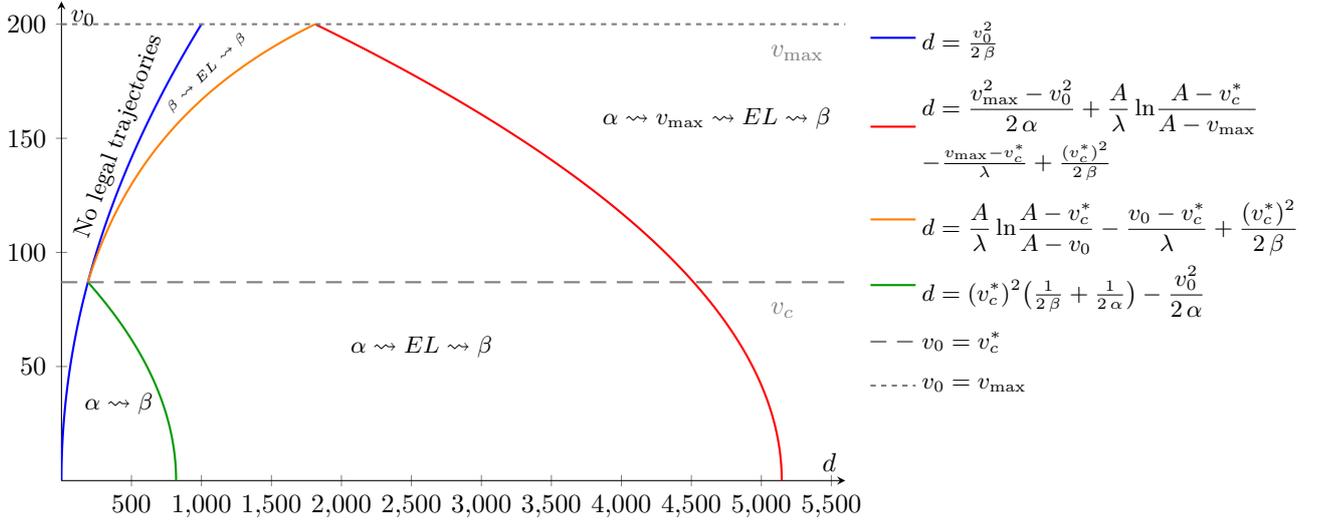
\begin{figure}[ht]
\centering
 \begin{tikzpicture}[scale=0.99]
  \pgfmathsetmacro{\vmax}{200}                   
  \pgfmathsetmacro{\alphaparam}{6}               
  \pgfmathsetmacro{\betaparam}{20}               
  \pgfmathsetmacro{\lambdaparam}{0.1}            
  \pgfmathsetmacro{\vc}{86.94452}                
  \pgfmathsetmacro{\A}{\vmax + \alphaparam/\lambdaparam}

  \begin{axis}[
    width=12cm, height=8cm,
    xlabel={$d$}, ylabel={$v_0$},
    xmin=0, xmax=5600, ymin=0, ymax=210,
    axis lines=middle,
    legend style={
      at={(1.02,0.98)}, anchor=north west,
      font=\small, draw=none,
      row sep=6pt
    },
    legend cell align=left,
  ]

    \addplot[thick,blue,domain=0:\vmax,samples=200]
      ({ x^2/(2*\betaparam) },{ x });
    \addlegendentry{$d=\tfrac{v_0^2}{2\,\beta}$}

    \addplot[thick,red,domain=0:\vmax,samples=200]
      ({ ( \vmax^2 - x^2 )/(2*\alphaparam)
         + \A/\lambdaparam*ln((\A-\vc)/(\A-\vmax))
         - (\vmax-\vc)/\lambdaparam
         + \vc^2/(2*\betaparam)
      },{ x });
    \addlegendentry{%
      \shortstack[l]{%
        $d=\displaystyle\frac{v_{\max}^2 - v_0^2}{2\,\alpha}
          + \frac{A}{\lambda}\ln\!\frac{A-v_c^*}{A-v_{\max}}$\\
        $-\frac{v_{\max}-v_c^*}{\lambda}
          + \frac{(v_c^*)^2}{2\,\beta}$%
      }%
    }

    \addplot[thick,orange,domain=\vc:\vmax,samples=200]
      ({ \A/\lambdaparam*ln((\A-\vc)/(\A-x))
         - (x-\vc)/\lambdaparam
         + \vc^2/(2*\betaparam)
      },{ x });
    \addlegendentry{%
      $d=\displaystyle\frac{A}{\lambda}\ln\!\frac{A-v_c^*}{A-v_0}
        - \frac{v_0-v_c^*}{\lambda}
        + \frac{(v_c^*)^2}{2\,\beta}$%
    }

    \addplot[thick,green!60!black,domain=0:\vc,samples=200]
      ({ \vc^2*(1/(2*\betaparam)+1/(2*\alphaparam))
         - x^2/(2*\alphaparam) },{ x });
    \addlegendentry{%
      $d=\displaystyle (v_c^*)^2\bigl(\tfrac1{2\,\beta}+\tfrac1{2\,\alpha}\bigr)
        - \frac{v_0^2}{2\,\alpha}$%
    }

    \addplot[
      thick, gray,
      dash pattern=on 6pt off 4pt
    ] coordinates {(0,\vc) (5600,\vc)};
    \addlegendentry{$v_0 = v_c^*$}

    \addplot[
      thick, gray,
      dash pattern=on 2pt off 2pt
    ] coordinates {(0,\vmax) (5600,\vmax)};
    \addlegendentry{$v_0 = v_{\max}$}

    \node[gray, anchor=north west] 
      at (axis cs:5000,{\vc-5}) {$v_c$};
    \node[gray, anchor=north west] 
      at (axis cs:5000,{\vmax-5}) {$v_{\max}$};

    \node[
      font=\small\bfseries,
      anchor=south west
    ] at (axis cs:100,25) {$\alpha \leadsto \beta$};

    \node[
      font=\tiny\bfseries,
      anchor=south west,
      rotate=45
    ] at (axis cs:800,155) {$\beta \leadsto EL \leadsto \beta$};

    \node[
      font=\small\bfseries,
      anchor=south west
    ] at (axis cs:2000,50) {$\alpha \leadsto EL \leadsto \beta$};

    \node[
      font=\small\bfseries,
      anchor=south west
    ] at (axis cs:3800,150) {$\alpha \leadsto v_{\mathrm{max}} \leadsto EL \leadsto \beta$};

    \node[
      font=\small,
      anchor=south west,
      rotate=70
    ] at (axis cs:260,100) {No legal trajectories};

  \end{axis}
\end{tikzpicture}
\caption{Phase-space diagram of $v_0$ vs.\ $d$, with $v_c^*<\vmax$ and $v_\beta>0$, and with parameters
$\lambda=0.1$, $\alpha=6$, $\beta=20$, $v_c^*\approx86.94$, and $v_{\max}=200$. Units do not reflect realistic physical values.}
\end{figure}
\end{condpics}

We may now also explore a situation in which $v_c^*>\vmax$ and $v_\beta>0$. Now that $v_c^*$ is outside of the legal range of velocities we can never have Euler--Lagrange in an optimal trajectory. Thus the problem simplifies greatly. $d\geq\frac{v_0^2}{2\beta}$ is still a necessary condition to not cross the traffic light. We can only follow $\alpha \leadsto \beta$ since there is no E--L and once $d$ is significantly large we have $\alpha \leadsto \vmax \leadsto \beta$. The boundary for this can be computed as the distance from $v_0$ to $\vmax$ at $\alpha$ acceleration, $\frac{\vmax^2-v_0^2}{2\alpha}$ added to the distance from $\vmax$ to stationary at $\beta$ deceleration, $\frac{\vmax^2}{2\beta}$. This results in the expression 
\[
d = \vmax^2\left(\frac{1}{2\alpha}+\frac{1}{2\beta}\right)-\frac{v_0^2}{2\alpha}.
\]

\begin{condpics}
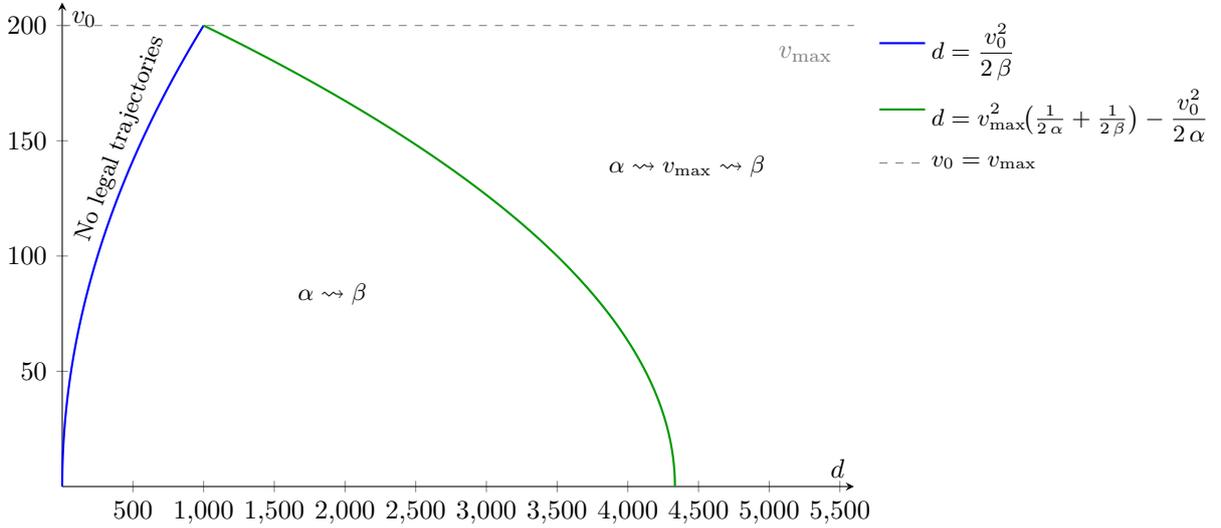
\begin{figure}[ht]
\centering
\begin{tikzpicture}
  \pgfmathsetmacro{\vmax}{200}                   
  \pgfmathsetmacro{\alphaparam}{6}               
  \pgfmathsetmacro{\betaparam}{20}               

  \begin{axis}[
    width=12cm, height=8cm,
    xlabel={$d$}, ylabel={$v_0$},
    xmin=0, xmax=5600, ymin=0, ymax=210,
    axis lines=middle,
    legend style={
      at={(1.02,0.98)}, anchor=north west,
      font=\small, draw=none,
      row sep=6pt      
    },
    legend cell align=left,
  ]

    \addplot[thick,blue,domain=0:\vmax,samples=200]
      ({ x^2/(2*\betaparam) },{ x });
    \addlegendentry{$d=\displaystyle\frac{v_0^2}{2\,\beta}$}

    \addplot[thick,green!60!black,domain=0:\vmax,samples=200]
      ({ \vmax^2*(1/(2*\alphaparam)+1/(2*\betaparam)) 
         - x^2/(2*\alphaparam) },
       { x });
    \addlegendentry{%
      $d=\displaystyle
        v_{\max}^2\!\bigl(\tfrac1{2\,\alpha}+\tfrac1{2\,\beta}\bigr)
        - \frac{v_0^2}{2\,\alpha}$%
    }

    \addplot[dashed,gray] coordinates {(0,\vmax) (5600,\vmax)};
    \addlegendentry{$v_0 = v_{\max}$}

    \node[gray, anchor=north west]
      at (axis cs:5000,{\vmax-5}) {$v_{\max}$};

    \node[
      font=\small\bfseries,
      anchor=south west
    ] at (axis cs:1600,75) {$\alpha \leadsto \beta$};

    \node[
      font=\small,
      anchor=south west,
      rotate=70
    ] at (axis cs:260,100) {No legal trajectories};

    \node[
      font=\small\bfseries,
      anchor=south west
    ] at (axis cs:3800,130) {$\alpha \leadsto v_{\mathrm{max}} \leadsto \beta$};

  \end{axis}
\end{tikzpicture}
\caption{Phase-space diagram of $v_0$ vs.\ $d$, with $v_c^*>\vmax$ and $v_\beta>0$,
and with parameters $\alpha=6$, $\beta=20$, and $v_{\max}=200$. Units are not from any standard unit system in which these are realistic physical values.}
\end{figure}
\end{condpics}

Finally, when $v_\beta \leq 0$ we do not have to worry about the value of $v_c^*$, since we never change from E--L to \(\be\). We still have the $\beta$ stopping boundary of $d\geq \frac{v_0^2}{2\beta}$. Our first boundary comes from the trajectory of travelling purely along an Euler--Lagrange curve from $v_0$ to stationary. Since no boundary conditions interfere this must be the unique optimal trajectory. Above this curve we will see $\beta \leadsto EL$ trajectories and below it $\alpha \leadsto EL$. This boundary can by derived as before in \eqref{eq:ELintegral} by integrating the E--L curve as it travels from $\vmax$ to $0$ to obtain

\[
d =\frac{A}{\lambda}\ln\frac{A}{A-v_0} - \frac{v_0}{\lambda}.
\]

Our final boundary condition comes from trajectories that accelerate to $\vmax$ at $\alpha$ and then after some time follow an E--L curve to stationary. The boundary is given by the distance covered accelerating to $\vmax$ from $v_0$ and then instantly decelerating down an E--L curve. This is given by 

\[
d = \frac{\vmax^2-v_0^2}{2\alpha} + \frac{A}{\lambda}\ln\frac{A}{A-\vmax} - \frac{\vmax}{\lambda}.
\]

This concludes all possible cases of optimal trajectory for the Exponential case. 

\begin{condpics}
\begin{figure}[ht]
\centering
\resizebox{\linewidth}{!}{%
\begin{tikzpicture}
  \pgfmathsetmacro{\vmax}{200}                   
  \pgfmathsetmacro{\alphaparam}{6}               
  \pgfmathsetmacro{\betaparam}{20}               
  \pgfmathsetmacro{\lambdaparam}{0.1}            
  \pgfmathsetmacro{\A}{\vmax + \alphaparam/\lambdaparam}

  \begin{axis}[
    width=12cm, height=8cm,
    xlabel={$d$}, ylabel={$v_0$},
    xmin=0, xmax=5600, ymin=0, ymax=210,
    axis lines=middle,
    legend style={
      at={(1.02,0.98)}, anchor=north west,
      font=\small, draw=none,
      row sep=6pt  
    },
    legend cell align=left,
  ]

    \addplot[thick,blue,domain=0:\vmax,samples=200]
      ({ x^2/(2*\betaparam) },{ x });
    \addlegendentry{$d=\displaystyle\frac{v_0^2}{2\,\beta}$}

    \addplot[thick,orange,domain=0:\vmax,samples=200]
      ({ \A/\lambdaparam*ln(\A/(\A - x)) - x/\lambdaparam },{ x });
    \addlegendentry{%
      $d=\displaystyle\frac{A}{\lambda}\ln\!\frac{A}{A - v_0}
        - \frac{v_0}{\lambda}$%
    }

    \addplot[thick,red,domain=0:\vmax,samples=200]
      ({ (\vmax^2 - x^2)/(2*\alphaparam)
         + \A/\lambdaparam*ln(\A/(\A - \vmax))
         - \vmax/\lambdaparam
      },{ x });
    \addlegendentry{%
      $d=\displaystyle\frac{v_{\max}^2 - v_0^2}{2\,\alpha}
        + \frac{A}{\lambda}\ln\!\frac{A - v_{\max}}{A}
        - \frac{v_{\max}}{\lambda}$%
    }

    \addplot[dashed,gray] coordinates {(0,\vmax) (5600,\vmax)};
    \addlegendentry{$v_0 = v_{\max}$}

    \node[gray, anchor=north west]
      at (axis cs:5000,{\vmax-5}) {$v_{\max}$};

    \node[
      font=\tiny\bfseries,
      anchor=south west,
      rotate=45
    ] at (axis cs:800,155) {$\beta \leadsto EL$};

    \node[
      font=\small\bfseries,
      anchor=south west
    ] at (axis cs:2000,75) {$\alpha \leadsto EL$};

    \node[
      font=\small\bfseries,
      anchor=south west
    ] at (axis cs:3800,150) {$\alpha \leadsto v_{\mathrm{max}} \leadsto EL$};

    \node[
      font=\small,
      anchor=south west,
      rotate=70
    ] at (axis cs:260,100) {No legal trajectories};

  \end{axis}
\end{tikzpicture}
}
\caption{Phase-space diagram of $v_0$ vs.\ $d$, with $v_\beta\le0$,
and parameters $\alpha=6$, $\beta=20$, $\lambda=0.1$, and $v_{\max}=200$.}
\end{figure}
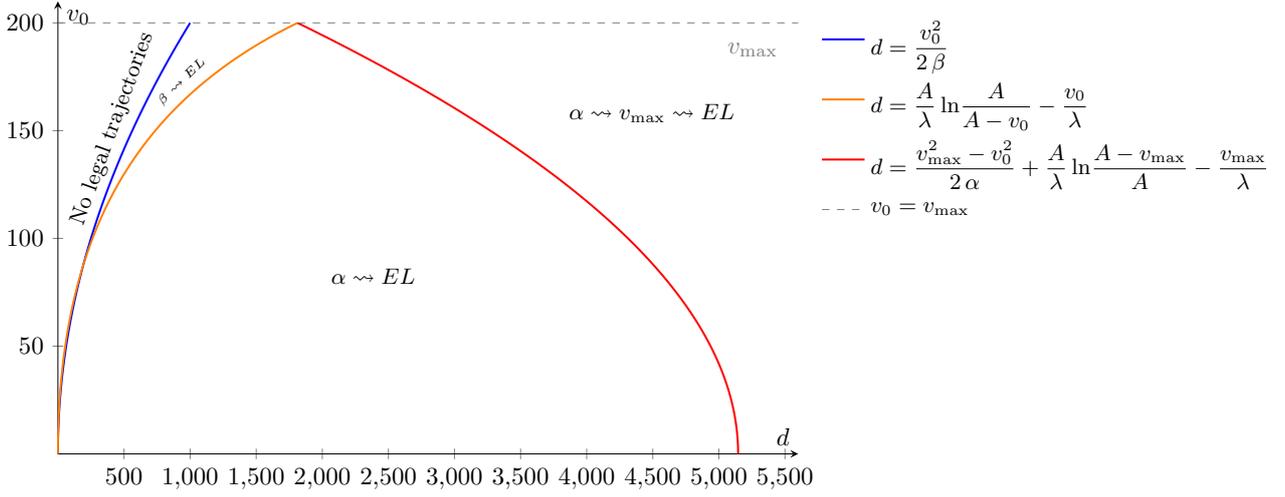
\end{condpics}

\section{Discussion}

We proved the existence of an optimal trajectory (of a car) with fixed initial location and velocity values, subject to constraints on its deceleration, acceleration and velocity. The objective function that we minimised was the expected arrival time to a distant destination, when an obstacle (the red light) at a certain distance is present until a random time with a known distribution. We formulated this objective function as a Lagrangian action integral, then showed that every minimising trajectory corresponds to a level of a pressure function that is associated to the Lagrangian, and moreover has a simply-described structure wherever it is above or below that value of the pressure. Our results are valid under natural assumptions on the distribution of the time at which the light will turn green. Using our general results, we worked out full details of the optimal path for two particular families of distribution: Uniform and Exponential.

Although the detailed description of the minimiser involved many cases to cover the full parameter space, we provided a rather intuitive tank-and-water picture which explains all possible cases. Our proofs used this interpretation, which is equivalent to the optimisation problem that we started with. This equivalence connects the natural Euler-Lagrange curves of the problem to level sets of the water surface.

The Exponential case revealed an unexpected feature of the optimal paths: when the Euler-Lagrange curves turn steeper than our constraint for maximum deceleration and, as a consequence, the trajectory ends with a phase of full deceleration, this switch happens well before the Euler-Lagrange reaches this maximum allowed slope. Instead, this last phase starts suddenly with a jump in deceleration (that is, slamming the brakes at a certain velocity). The phenomenon is easily understood via the tank-and-water interpretation; the full deceleration phase at the end plays the role of a moveable wall in the inhomogeneous pressure landscape, and it takes a position of a mechanical equilibrium to equalise the hydrostatic forces acting on it.

The methods of this paper seem reasonably robust and we expect them to be suitable for further generalisations. This could involve minimising a combination of energy and arrival time, or considering further classes of distributions for the duration of the red phase of the traffic light. It would also be interesting to consider the case where the duration of the green phase following the initial red phase of the traffic light is not guaranteed to be long enough for the car to pass the light.

\section*{Acknowledgements}
We thank Gyula D\'avid for illuminating discussions on reparametrising action integrals. E.\ C.\ and A.\ T.\ were supported by the Heilbronn Institute for Mathematical Research. M.\ B.\ was partially supported by the EPSRC EP/W032112/1 Standard Grant of the UK. This study did not involve any underlying data.
\section*{Statement on AI use}
AI was used to help create the figures and was used by the authors during the revision to ensure that all the referees' comments had been addressed. All the mathematical arguments and writing are due to the authors.

\end{document}